\documentclass[11pt,a4paper]{article}
\usepackage[margin=2.1cm]{geometry}
\usepackage{amsmath,amsthm,amsfonts,amssymb,amscd,cite,graphicx}
\usepackage{latexsym}
\usepackage{enumitem}

\usepackage[usenames, dvipsnames]{color}
\definecolor{mygray}{gray}{0.6}

\usepackage{titlesec}
\titleformat{\section}
{\normalfont\fontsize{12}{15}\bfseries}{\thesection}{1em.}{}

\newtheorem{proposition}{Proposition}[section]

\newtheorem{lemma}{Lemma}[section]

\newtheorem{theorem}{Theorem}[section]

\allowdisplaybreaks[4]

\let\oldbibliography\thebibliography
\renewcommand{\thebibliography}[1]{%
  \oldbibliography{#1}%
  \setlength{\itemsep}{-2pt}%
}

\begin{document}

\baselineskip=0.20in

\makebox[\textwidth]{%
\hglue-15pt
\begin{minipage}{0.6cm}	
\vskip9pt
\end{minipage} \vspace{-\parskip}
\hfill
}
\vskip36pt

\noindent
{\large \bf Avoidance of vincular patterns by Catalan words}\\

\noindent
Toufik Mansour and Mark Shattuck\\

\noindent
\footnotesize {\it Department of Mathematics, University of Haifa, 3498838 Haifa, Israel\\
Email: tmansour@univ.haifa.ac.il}\\

\noindent
\footnotesize {\it Department of Mathematics, University of Tennessee,
37996 Knoxville, TN, USA\\
Email: mshattuc@utk.edu}\\

\noindent
(\footnotesize Received: Day Month 201X. Received in revised form: Day Month 201X. Accepted: Day Month 201X. Published online: Day Month 201X.)\\

\setcounter{page}{1} \thispagestyle{empty}

\baselineskip=0.20in

\normalsize

\begin{abstract}
Let $\mathcal{C}_n$ denote the set of words $w=w_1\cdots w_n$ on the alphabet of positive integers satisfying $w_{i+1}\leq w_i+1$ for $1 \leq i \leq n-1$ with $w_1=1$.  The members of $\mathcal{C}_n$ are known as Catalan words and are enumerated by the $n$-th Catalan number $C_n$.  The problem of finding the cardinality of various avoidance classes of $\mathcal{C}_n$ has been an ongoing object of study, and members of $\mathcal{C}_n$ avoiding one or two classical or a single consecutive pattern have been enumerated.  In this paper, we extend these results to vincular patterns and seek to determine the cardinality of each avoidance class corresponding to a pattern of type (1,2) or (2,1).  In several instances, a simple explicit formula for this cardinality may be given.  In the more difficult cases, we find only a formula for the (ordinary) generating function which enumerates the class in question.  We make extensive use of functional equations in establishing our generating function results. \medskip

\medskip


\end{abstract}

\section{Introduction}

Let $\tau=\tau_1\cdots \tau_m$ denote a sequence in $[\ell]=\{1,\ldots,\ell\}$ for some $\ell\geq1$ that contains each letter in $[\ell]$ at least once and let $\pi=\pi_1\cdots\pi_n$ be a positive integral sequence, where $n \geq m \geq 1$. Then $\pi$ is said to \emph{contain} $\tau$ if there exist indices $1 \leq i_1<\cdots<i_m\leq n$ such that $\pi_{i_j} \,x\, \pi_{i_k}$ if and only if $\tau_j \,x\,\tau_k$ for all $j,k \in [m]$ and each $x \in\{<,>,=\}$,  That is, $\pi$ contains $\tau$ if there exists a subsequence of $\pi$ that is order-isomorphic to $\tau$ and is said to \emph{avoid} $\tau$ otherwise.  In this context, $\tau$ is often referred to as a \emph{pattern}.  The pattern avoidance problem is one that has been studied extensively in enumerative combinatorics on a variety of discrete structures, starting with permutations.

The concept of pattern avoidance can be generalized as follows.  Suppose one decomposes the pattern $\tau$ into nonempty sections as $\tau=\tau^{(1)}\cdots \tau^{(k)}$ for some $k \geq 1$. Then $\pi$ contains an occurrence of $\tau=\tau^{(1)}\text{-}\cdots\text{-}\tau^{(k)}$ as a \emph{vincular} (or \emph{dashed}) pattern if there exists a subsequence $\alpha$ of $\pi$ that is order-isomorphic to $\tau$ such that, for each $p\in [m-1]$, if $\tau_p$ and $\tau_{p+1}$ both belong to the same section $\tau^{(i)}$ of $\tau$ for some $i \in [k]$, then the $p$-th and $(p+1)$-st entries of $\alpha$ correspond to adjacent letters of $\pi$.  That is, letters within the groups separated by dashes must be adjacent within an occurrence of $\tau$ in $\pi$. If $a_i=|\tau^{(i)}|$ for $1 \leq i \leq k$, then the vincular pattern $\tau=\tau^{(1)}\text{-}\cdots\text{-}\tau^{(k)}$ is said to be of \emph{type} $(a_1,\ldots,a_k)$.  We refer the reader to the authoritative text by Kitaev \cite[Chapter\,7]{Kit} for a full discussion of these and other kinds of patterns.

For example, the sequence $\pi=123411232$ contains one occurrence of the pattern $\tau=12\text{-}3\text{-}21$ of type (2,1,2), as witnessed by the subsequence 23432.  On the other hand, $\pi$ avoids the pattern $13\text{-}2$ of type (2,1), though it is seen to contain subsequences isomorphic to 132.  Note that the cases in which each $\tau^{(i)}$ contains a single letter or in which $k=1$ correspond to what are known as the \emph{classical} and \emph{consecutive} patterns, respectively.  Thus, any isomorphic subsequence of $\pi$ gives rise to an occurrence of a classical pattern, whereas for a consecutive pattern (sometimes called a \emph{subword}), the entries must correspond to a string of adjacent letters of $\pi$ and therefore it imposes the least restriction with regard to the avoidance of which.

A \emph{Catalan} word $\pi=\pi_1\cdots \pi_n$ is a sequence of positive integers satisfying $\pi_{i+1}\leq \pi_i+1$ for $1 \leq i \leq n-1$, with $\pi_1=1$. Let $\mathcal{C}_n$ denote the set of Catalan words of length $n$. For example, if $n=4$, then
$$\mathcal{C}_4=\{1111,1121,1211,1221,1231,1112,1122,1212,1222,1232,1123,1223,1233,1234\}.$$
It is well known that $|\mathcal{C}_n|=C_n$ for all $n\geq0$, where $C_n=\frac{1}{n+1}\binom{2n}{n}$ is the $n$-th Catalan number; see, e.g., \cite[Exercise~80]{Stan}. This fact may be realized quickly as follows.  Let $\mathcal{D}_n$ denote the set of \emph{Dyck} paths of semi-length $n$, i.e., lattice paths that use $u=(1,1)$ and $d=(1,-1)$ steps going from $(0,0)$ to $(2n,0)$ and never dipping below the $x$-axis, which is a fundamental structure enumerated by $C_n$. Given $\pi=\pi_1\cdots \pi_n \in \mathcal{C}_n$, let $\imath(\pi)$ denote the member of $\mathcal{D}_n$ whose $j$-th $u$ from the left terminates at height $\pi_j$ for each $j \in [n]$.  For example, if $\pi=12342332 \in \mathcal{C}_8$, then $\imath(\pi)=u^4d^3u^2dud^2ud^2\in \mathcal{D}_8$.   It is apparent that $\imath$ provides a bijection between $\mathcal{A}_n$ and $\mathcal{D}_n$ such that the final letter of $\pi$ equals the number of $d$'s within the terminal run of $d$'s in $\imath(\pi)$ for all $\pi$. Thus, the results below counting certain restricted subsets of $\mathcal{C}_n$ may also be viewed equivalently as enumerative formulas for the corresponding subsets of $\mathcal{D}_n$ under $\imath$.

The study of pattern avoidance in Catalan words was initiated by Baril et al. in \cite{BKV2}, where every case of a single three-letter classical pattern was treated and the descent distribution on each corresponding avoidance class was found. These results were later extended to two classical patterns (see \cite{BKV} or \cite[Chapter\,3]{Kh}), subwords \cite{RRO} and partial order patterns \cite{BRam}.  Moreover, the distributions of several parameters have also been considered on Catalan words, represented geometrically such as bargraphs, such as  semi-perimeter \cite{CMR,MSh2}, area \cite{CMR,MSh2} and the number of interior lattice points \cite{MRT}.  Further, in \cite{Sha}, the joint distributions of one or more subwords on $\mathcal{C}_n$ were studied and explicit formulas for the generating functions were found, and in \cite{MVaj}, Catalan words arose in connection with the exhaustive generation of Gray codes for growth-restricted sequences.

In this paper, we consider the problem of enumerating the members of $\mathcal{C}_n$ avoiding a single vincular pattern of length three, in response to a general question raised by Baril et al. in \cite{BKV2} concerning the vincular pattern avoidance problem on $\mathcal{C}_n$ and, in particular, the case of three-letter patterns.  Let $\mathcal{C}_n(\tau)$ denote the subset of $\mathcal{C}_n$ for $n \geq1$  whose members avoid the pattern $\tau$ and let $c_n(\tau)=|\mathcal{C}_n(\tau)|$.  Here, we determine an explicit formula for $c_n(\tau)$ or its generating function $\sum_{n\geq1}c_n(\tau)t^n$ for each vincular pattern of the form $x\text{-}yz$ or $xy\text{-}z$, answering the question raised in \cite{BKV2}.  We remark that the Catalan words treated in \cite{MSh1} correspond to a different class of non-negative integral sequences, which are in fact enumerated by $C_{n-1}$, and that are characterized by their satisfying the growth requirement $\pi_{i+1} \geq \pi_i-1$ for $1 \leq i <n$ together with the condition that the leftmost occurrence of each positive letter $k$ has at least one $k-1$ occurring somewhere both to its left and its right.

The organization of this paper is as follows.  In the next section, we enumerate the members of $\mathcal{C}_n(\tau)$, where $\tau$ is a pattern of the type (1,2).  Our work is shortened somewhat in this regard by noting that several cases of avoiding (1,2) are logically equivalent on $\mathcal{C}_n$ to the avoidance of the corresponding classical pattern of length three obtained by ignoring the adjacency requirement, and hence the results in these cases follow from those in \cite{BKV2}.  Direct combinatorial proofs are given for the patterns 1-22, 1-32 and 2-31 in subsection 2.1, with $c_n(\tau)$ in these cases corresponding to well-known sequences from the OEIS \cite{Sloane}.  The remaining patterns in (1,2) are apparently more difficult and we only determine a formula for the generating function $\sum_{n\geq1}c_n(\tau)t^n$ in each case.  For the patterns 2-21 and 3-21 considered in 2.2, we make extensive use of functional equations to obtain our results and employ various techniques such as iteration and the \emph{kernel} method (see, e.g., \cite{HM}).  By contrast, for the cases 1-11 and 2-11 treated in 2.3, we utilize the \emph{symbolic} method (see, e.g., \cite{FS}) and are able to find expressions for the generating functions in terms of Chebyshev polynomials.

A similar treatment is afforded the patterns of type (2,1) in the third section.  For the cases 22-1 and 32-1 in subsection 3.1, a direct enumeration may be given.  For the others, we make use of functional equations to obtain the generating function formulas.  In the case of 21-1 covered in 3.5, an elegant formula in terms of an infinite continued fraction arises somewhat unexpectedly.  For several cases of either type (1,2) or (2,1), we must refine the counting sequence $c_n(\tau)$ by considering one or more parameters (specific to the avoidance class at hand) on $\mathcal{C}_n(\tau)$ so as to be able to write a recurrence that enumerates the class.  A variety of such parameters are utilized in this way, including those tracking the largest letter, last letter, number of $1$'s, smallest descent bottom and position of the first level.

The results of this paper are summarized in Tables \ref{tab1} and \ref{tab2} below.  Note that every Catalan word avoids the patterns 2-13 and 13-2 by virtue of the condition $\pi_{i+1}\leq \pi_i+1$ for all $i$, and hence these cases are trivial.

\begin{table}[htp]
\begin{center}
\begin{tabular}{|l|l|l|}\hline
  $\tau$ & $\{c_n(\tau)\}_{n=1}^{10}$ & Reference\\\hline\hline
$1\text{-}12$, $1\text{-}21$, $1\text{-}23$&1,2,4,8,16,32,64,128,256,512&Proposition~\ref{(1,2)prop}\\\hline
$1\text{-}22$&1,2,4,9,21,51,127,323,835,2188&Theorem~\ref{RcasesTh2}\\\hline
$1\text{-}11$&1,2,4,10,25,66,179,495,1390,3951&Theorem~\ref{1-11th1}\\\hline
$1\text{-}32$, $2\text{-}12$&1,2,5,13,34,89,233,610,1597,4181&Theorem~\ref{RcasesTh1} and Proposition~\ref{(1,2)prop}\\\hline
$2\text{-}21$&1,2,5,13,34,90,240,643,1728,4654&Theorem~\ref{th2-21}\\\hline
$2\text{-}31$&1,2,5,13,35,96,267,750,2123,6046&Theorem~\ref{RcasesTh2}\\\hline
$2\text{-}11$&1,2,5,13,36,103,302,901,2724,8321&Theorem~\ref{2-11th}\\\hline
$3\text{-}12$&1,2,5,14,41,122,365,1094,3281,9842&Proposition~\ref{(1,2)prop}\\\hline
$3\text{-}21$&1,2,5,14,41,122,366,1104,3344,10162&Theorem~\ref{th3-21}\\\hline
$2\text{-}13$&1,2,5,14,42,132,429,1430,4862,16796&Trivial\\\hline
\end{tabular}
\end{center}
\caption{Number of Catalan words of length $n=1,2,\ldots,10$ that avoid a pattern of form $x\text{-}yz$}\label{tab1}
\end{table}
\begin{table}[htp]
\begin{center}
\begin{tabular}{|l|l|l|}\hline
  $\tau$ & $\{c_n(\tau)\}_{n=1}^{10}$&Reference \\\hline\hline
$12\text{-}2$&1,2,4,7,11,16,22,29,37,46&Proposition~\ref{(2,1)prop}\\\hline
$12\text{-}1$, $12\text{-}3$&1,2,4,8,16,32,64,128,256,512&Proposition~\ref{(2,1)prop}\\\hline
$11\text{-}2$&1,2,4,9,22,56,148,400,1102,3079&Theorem~\ref{11-2th}\\\hline
$11\text{-}1$&1,2,4,10,25,66,179,495,1390,3951&Proposition~\ref{11-1prop}\\\hline
$23\text{-}1$&1,2,5,13,34,89,233,610,1597,4181&Proposition~\ref{(2,1)prop}\\\hline
$21\text{-}1$&1,2,5,13,35,96,267,750,2122,6036&Theorem~\ref{21-1th}\\\hline
$22\text{-}1$&1,2,5,13,35,96,267,750,2123,6046&Theorem~\ref{22-1/32-1}\\\hline
$21\text{-}2$&1,2,5,13,35,97,274,784,2265,6593&Theorem~\ref{21-2thm}\\\hline
$32\text{-}1$&1,2,5,14,41,122,365,1094,3281,9842&Theorem~\ref{22-1/32-1}\\\hline
$31\text{-}2$&1,2,5,14,41,123,375,1157,3602,11291&Theorem~\ref{31-2th}\\\hline
$21\text{-}3$&1,2,5,14,41,123,375,1157,3603,11303&Theorem~\ref{21-3thm}\\\hline
$13\text{-}2$&1,2,5,14,42,132,429,1430,4862,16796&Trivial\\\hline
\end{tabular}
\end{center}
\caption{Number of Catalan words of length $n=1,2,\ldots,10$ that avoid a pattern of form $xy\text{-}z$}\label{tab2}
\end{table}

\section{Patterns of the type $(1,2)$}

In this section, we seek to determine the sequences $c_n(\tau)$, where $\tau$ is a vincular pattern of the form (1,2), either explicitly or in terms of generating functions.  Our work in this regard is shortened by observing that several of the (1,2) patterns are logically equivalent for Catalan words to the classical pattern of the same length obtained by removing the adjacency requirement of the last two letters within an occurrence.

For example, we have $c_n(3\text{-}12)=c_n(3\text{-}1\text{-}2)$ for all $n \geq 1$. To realize this, suppose $\pi=\pi_1\cdots\pi_n \in \mathcal{C}_n$ contains an occurrence of $3\text{-}1\text{-}2$ as witnessed by the subsequence $\pi_i=a, \pi_j=b, \pi_k=c$, where $b<c<a$.  We show that $\pi$ must contain an occurrence of $3\text{-}12$.  To do so, let $r_0$ be the smallest index $r>j$ such that $\pi_r=c$.  Then $\pi_k=c$ with $k>j$ implies $r_0$ exists with $j<r_0\leq k$.  We consider now cases on $\pi_{r_0-1}$ and first suppose $\pi_{r_0-1}>c$.  Note that $\pi$ being a Catalan word implies it can have only unit increases between adjacent elements, and thus $\pi_j=b<c$ implies there must exist some index $r \in[j+1,r_0-2]$ such that $\pi_r=c$.  But this contradicts the minimality of $r_0$, so we must have $\pi_{r_0-1}<c$, i.e, $\pi_{r_0-1}=c-1$ (note $\pi_{r_0-1}\neq c$, again by the minimality).  But then $\pi$ witnesses an occurrence of $3\text{-}12$ with the subsequence $\pi_i=a,\pi_{r_0-1}=c-1,\pi_{r_0}=c$, as desired.  This implies that the patterns $3\text{-}12$ and $3\text{-}1\text{-}2$ are equivalent on $\mathcal{C}_n$, as claimed.

By comparable reasoning, one has that the vincular patterns $1\text{-}12$, $1\text{-}21$, $1\text{-}23$ and $2\text{-}12$ are equivalent to the analogous classical patterns.  Let $F_n=F_{n-1}+F_{n-2}$ for $n \geq 2$ denote the $n$-th Fibonacci number, with $F_0=0$ and $F_1=1$; see, e.g.,  \cite[A000045]{Sloane} or \cite{Vaj}.  By the results from \cite{BKV2}, we then have the following.

\begin{proposition}\label{(1,2)prop}
If $n\geq1$, then $c_n(1\text{-}12)=c_n(1\text{-}21)=c_n(1\text{-}23)=2^{n-1}$, $c_n(2\text{-}12)=F_{2n-1}$ and $c_n(3\text{-}12)=\frac{3^{n-1}+1}{2}$.
\end{proposition}

\subsection{The cases 1-22, 1-32 and 2-31}

We start with the following result for the pattern 1-32.

\begin{theorem}\label{RcasesTh1}
If $n \geq 1$, then $c_n(1\text{-}32)=F_{2n-1}$.
\end{theorem}
\begin{proof}
Let $\Omega_n$ denote the set of marked words $w=w_1\cdots w_n$ wherein $w$ is a weakly increasing Catalan word in which runs of letters are marked such that (i) the initial run of $1$'s is always marked and (ii) no two adjacent runs of letters may be marked.  Let $a_n=|\Omega_n|$ and we first show $a_n=F_{2n-1}$ for $n \geq 1$.  Then $a_1=1$ and $a_2=2$, and we show $a_n=3a_{n-1}-a_{n-2}$ for $n \geq 3$. Note first that there are $a_{n-1}$ members of $\Omega_n$ whose final run is of length at least two, upon appending a copy of the final letter to a member of $\Omega_{n-1}$.  Further, there are also $a_{n-1}$ members of $\Omega_n$ ending in a run of length one such that this run is of the opposite status concerning whether or not it is marked than the run directly preceding it.  To see this, append $m+1$ to a member of $\Omega_{n-1}$ whose last (= largest) letter is $m$ for each $m \geq 1$, with $m+1$ being marked if and only if the run containing $m$ is unmarked. To complete the proof of the recurrence, it suffices to show that there are $a_{n-1}-a_{n-2}$ members of $\Omega_n$ for $n \geq 3$ ending in two or more unmarked runs, with the last run of length one.   Let $\Omega_n'$ and $\Omega_n''$ denote the subsets of $\Omega_n$ ending in an unmarked or a marked run, respectively. Then there are $a_{n-1}$ members of $\Omega_n''$, which may be realized by appending a marked $m+1$ to $\lambda \in \Omega_{n-1}'$ or appending $m$ to $\lambda \in \Omega_{n-1}''$, where $m$ denotes the largest letter of $\lambda$ in either case. Hence, by subtraction, there are $a_{n-1}-a_{n-2}$ members of $\Omega_{n-1}'$, and appending an unmarked $m+1$ as before yields the desired members of $\Omega_n$.

We now define a bijection between $\Omega_n$ and $\mathcal{C}_n(1\text{-}32)$ for all $n \geq 1$, which will imply the second result.  Represent $\pi \in \Omega_n$ by $\pi=\pi^{(1)}\cdots\pi^{(s)}$ for some $s \geq 1$, where each $\pi^{(i)}$ starts with a marked run and contains no other marked runs, with $\min(\pi^{(i+1)})=\max(\pi^{(i)})+1$ for $1 \leq i \leq s$.  Let $\min(\pi^{(i)})=j_i$ for each $i \in [s]$ so that $\pi \in \Omega_n$ implies $j_{i+1}>j_i+1$ for all $i$.  Let $f(\pi)$ be the sequence obtained from $\pi$ by replacing the section $\pi^{(i)}$ for each $i \in [s]$ with $\pi^{(i)}-(j_i-1)$, i.e, subtract $j_i-1$ from each letter in $\pi^{(i)}$.  Then $\pi$ a weakly increasing Catalan word marked as described implies $f(\pi) \in \mathcal{C}_n(1\text{-}32)$ for all $\pi$ as each descent bottom in $f(\pi)$ is seen to be $1$.  One may verify that the mapping $f$ yields the desired bijection between $\Omega_n$ and $\mathcal{C}_n(1\text{-}32)$.
\end{proof}

Let $M_n$ denote the $n$-th Motzkin number for $n \geq0$, see, e.g., \cite[A001006]{Sloane}. Recall that $M_n$ enumerates the set $\mathcal{M}_n$  of lattice paths from $(0,0)$ to $(n,0)$ using $u$, $d$ and $h=(1,0)$ steps that never go below the $x$-axis (termed \emph{Motzkin} paths).  Let $L_n$ for $n \geq0$ denote the $n$-th term of the sequence A005773 from \cite{Sloane}.  Note that $L_n$ enumerates the set of lattice paths from $(0,0)$ to the line $x=n-1$ using $u$, $d$ and $h$ steps that never go below the $x$-axis (termed \emph{Motzkin left-factors}, see \cite[p.\,111]{AS}).

The following result connects the sequences $M_n$ and $L_n$ to vincular pattern avoidance by Catalan words.

\begin{theorem}\label{RcasesTh2}
If $n \geq 1$, then $c_n(1\text{-}22)=M_n$ and $c_n(2\text{-}31)=L_n$.
\end{theorem}
\begin{proof}
Let $\mathcal{C}_n'$ denote the subset of $\mathcal{C}_n$ whose members contain no two equal adjacent entries, i.e., contain no levels.  Let $\mathcal{D}_n(\rho)$ denote the subset of $\mathcal{D}_n$ whose members avoid the string of steps $\rho$.  Then we have that the subset $\mathcal{C}_n'$ of $\mathcal{C}_n$ corresponds under $\imath$ to $\mathcal{D}_n(udu)$.  We first define a bijection $\alpha_n$ between $\mathcal{D}_n(udu)$ and $\mathcal{M}_{n-1}$ for each $n\geq1$, which will be needed in what follows.  If $n=1$, let $\alpha_1(ud)$ be the empty lattice path of length zero.  If $n \geq 2$, let $\tau \in \mathcal{D}_n(udu)$ be decomposed as $\tau=\tau^{(1)}\cdots \tau^{(r)}$ for some $ r\geq 1$, where $\tau^{(j)}$ for each $j \in [r]$ is of the form $\tau^{(j)}=u\rho^{(j)}d$ with $\rho^{(j)}$ a possibly empty Dyck path (the $\tau^{(j)}$ are what are often referred to as the \emph{units} of the Dyck path $\tau$, with the same terminology applied to members of $\mathcal{M}_n$).

Let $k_j=|\tau^{(j)}|-1$ for $1 \leq j \leq r$, where $|\rho|$ denotes the semi-length (i.e., half the total number of steps) of a Dyck path $\rho$.  Note that $\tau$ avoiding $udu$ implies $k_j \geq 1$ for $j \in [r-1]$, with $k_r \geq 0$.  We then define $\alpha_n$ for $n\geq 2$ by letting
$$\alpha_n(\tau)=u\alpha_{k_1}(\rho^{(1)})d\cdots u\alpha_{k_{r-1}}(\rho^{(r-1)})d\sigma,$$
where $\sigma$ is empty if $k_r=0$ and is given by $h\alpha_{k_r}(\rho^{(r)})$ if $k_r>0$.  For example, this yields $\alpha_2(u^2d^2)=h$ and $\alpha_3(u^2d^2ud)=ud$, $\alpha_3(u^3d^3)=h^2$ and $\alpha_4(u^2d^2u^2d^2)=udh$, $\alpha_4(u^3d^3ud)=uhd$, $\alpha_4(u^3d^2ud^2)=hud$, $\alpha_4(u^4d^4)=h^3$, which covers all cases for $2 \leq n \leq 4$.  One may verify that $\alpha_n$ defines a bijection between $\mathcal{D}_n(udu)$ and $\mathcal{M}_{n-1}$ for all $n \geq 1$. Note that $\alpha_n$ may be reversed by considering the leftmost $h$ step of height zero, if it exists, within a member $\lambda$ of $\mathcal{M}_{n-1}$ as well as any units of $\lambda$ occurring to the left of this $h$ (which determine all units but the last of $\alpha_n^{-1}(\lambda)$).

Observe that $\mathcal{C}_n(1\text{-}22)$ corresponds under $\imath$ to the subset $\mathcal{D}_n^*$ of $\mathcal{D}_n$ consisting of those members which avoid all occurrences of $udu$ except for possibly those in which both $u$ steps end at height one.  We define a bijection $\beta_n$ between $\mathcal{D}_n^*$ for each $n \geq 1$  and $\mathcal{M}_n$ and taking the composition $\beta_n \circ \imath$ will yield the desired bijection between $\mathcal{C}_n(1\text{-}22)$ and $\mathcal{M}_n$.  To define $\beta_n$, let $\tau \in \mathcal{D}_n^*$ be decomposed into units as $\tau=\tau^{(1)}\cdots \tau^{(r)}$ like before and note that $k_j=0$ is now also possible for $j \in [r-1]$ (as $udu$'s flush with the $x$-axis are permitted in members of $\mathcal{D}_n^*$). If $k_j \geq 1$ for some $j \in [r]$, then we replace $\tau^{(j)}$ with $u\alpha_{k_j}(\rho^{(j)})d$, where $k_j$ and $\rho^{(j)}$ are as before.  If $k_j=0$, i.e., $\tau^{(j)}=ud$, then replace $\tau^{(j)}$ with $h$.  Let $\beta_n(\tau)$ denote the lattice path resulting when one makes all of the replacements as described above.  One may verify $\beta_n(\tau)\in \mathcal{M}_n$ for all $\tau$.  Further, one has that $\beta_n$ is reversible, and hence a bijection as desired, upon considering the units and any $h$ steps of height zero within a member of $\mathcal{M}_n$.  This completes the proof of the first equality.

For the second, note that $\pi=\pi_1\cdots\pi_n \in \mathcal{C}_n(2\text{-}31)$ must satisfy $|\pi_{i+1}-\pi_i|\leq 1$ for $1 \leq i \leq n-1$.  To see this, note that if $\pi_j\pi_{j+1}=ab$ with $a \geq b+2$ for some $j \in [n-1]$, then there must be an occurrence of 2-31, as witnessed by the subsequence $\pi_k\pi_j\pi_{j+1}$, where $k<j$ denotes the position of the leftmost occurrence of the letter $a-1$.  Thus, members of $\mathcal{C}_n(2\text{-}31)$ correspond to the set $\mathcal{T}_n$ of \emph{smooth} Catalan sequences of length $n$.  That $|\mathcal{T}_n|=L_n$ for $n \geq 1$ was shown in \cite[Theorem 2.1]{MSh2}, where a more general result concerning the generating function was deduced, which implies the second equality.
\end{proof}

\subsection{The cases 2-21 and 3-21}

One can find recurrences enumerating the members of $\mathcal{C}_n(\tau)$, where $\tau$ is 2-21 or 3-21, by refining the number of avoiders according to the same pair of parameters in either case.  Let $u_n(m,a)$ denote the cardinality of the subset of $\mathcal{C}_n(2\text{-}21)$ whose members have largest letter $m$ and last letter $a$, where $n \geq 1$ and $1 \leq a \leq m \leq n$, and put $u_n(m,a)=0$ otherwise.  We define the array $v_n(m,a)$ analogously in conjunction with the pattern 3-21.

We have the following recurrence for the array $u_n(m,a)$.

\begin{lemma}\label{2-21lem1}
If $n \geq 3$ and $2 \leq m \leq n$, then
\begin{align}
u_n(m,a)&=u_{n-1}(m,a-1)+u_{n-1}(m,a)+u_{n-2}(m-1,m-1), \quad 1 \leq a \leq m-1, \label{2-211rec1x}\\
u_n(m,m)&=u_{n-1}(m-1,m-1)+u_{n-1}(m,m-1)+u_{n-1}(m,m), \label{2-211rec2x}
\end{align}
with $u_n(1,1)=1$ for all $n \geq 1$ and $u_2(2,1)=0$, $u_2(2,2)=1$.
\end{lemma}
\begin{proof}
Let $\mathcal{U}_{n,m,a}$ denote the subset of $\mathcal{C}_n(2\text{-}21)$ enumerated by $u_n(m,a)$.  Suppose $\pi \in \mathcal{U}_{n,m,a}$, where $n \geq 3$ and $1 \leq a \leq m-1$.  Let $\ell$ denote the penultimate letter of $\pi$.  If $\ell \leq a$, then only $\ell =a-1$ or $\ell=a$ are possible since $\pi \in \mathcal{C}_n$ (where the former case only applies if $a>1$).  This yields $u_{n-1}(m,a-1)$ and $u_{n-1}(m,a)$ possibilities for $\pi$, respectively.  Note that $a<\ell<m$ is not possible, for otherwise there would be an occurrence of 2-21 in which the letters corresponding to the `2' correspond to the first and last occurrences of $\ell$.  Finally, if $\ell=m$, then there can be no other occurrences of $m$ in $\pi$ as $a<m$.  That is, $\pi=\pi'ma$, where $\pi' \in \mathcal{U}_{n-2,m-1,m-1}$.  Thus, there are $u_{n-2}(m-1,m-1)$ possibilities for $\pi$ if $\ell=m$ and combining with the prior cases yields \eqref{2-211rec1x}.  On the other hand, if $\pi \in \mathcal{U}_{n,m,m}$, where $m \geq 2$,
then there are $u_{n-1}(m,m-1)+u_{n-1}(m-1,m-1)$ possibilities if $\ell=m-1$, the second term accounting for cases in which the largest letter occurs only in the final position of $\pi$, and $u_{n-1}(m,m)$ possibilities if $\ell=m$, which implies \eqref{2-211rec2x}.  Clearly, $\mathcal{U}_{n,1,1}$ contains only the sequence $1^n$ for all $n \geq 1$, with $\mathcal{U}_{2,2,1}=\varnothing$ and $\mathcal{U}_{2,2,2}=\{12\}$, which implies the initial conditions.
\end{proof}

Define the distribution polynomial $\widetilde{u}_n(x,y)=\sum_{m=1}^n\sum_{a=1}^mu_n(m,a)x^{m-a}y^m$ for $n \geq 1$ and its generating function $u(t;x,y)=\sum_{n\geq1}\widetilde{u}_n(x,y)t^n$.  Define $\widetilde{v}_n(x,y)$ and $v(t;x,y)$ analogously in conjunction with the pattern 3-21.  There is the following explicit formula for $u(t;x,y)$ in the case $x=y=1$.

\begin{theorem}\label{th2-21}
We have
\begin{equation}\label{th2-21e1}
\sum_{n\geq1}c_n(2\text{-}21)t^n=\frac{t}{1-2t}-\sum_{i\geq1}\frac{(-1)^it^{\frac{i(i+5)}{2}}(1-2t)\sum_{j=1}^i\frac{(1-t)^j}{(1-2t)(1-t)^j-(1-t)^2t^j}}{\prod_{j=1}^i\left((1-2t)(1-t)^j-(1-t)^2t^j\right)}.
\end{equation}
\end{theorem}
\begin{proof}
Define $u_n(m;x)=\sum_{a=1}^mu_n(m,a)x^{m-a}$ for $n \geq 1$ and $1 \leq m \leq n$.  Then recurrences \eqref{2-211rec1x} and \eqref{2-211rec2x} can be written for $n \geq 3$ as
\begin{align}
u_n(m;x)&=\frac{1}{x}(u_{n-1}(m;x)-u_{n-1}(m;0))+u_{n-1}(m;x)+\frac{x-x^m}{1-x}u_{n-2}(m-1;0)\notag\\
&\quad+u_{n-1}(m-1;0), \quad 2 \leq m \leq n,\label{un(m;x)}
\end{align}
with $u_n(1;x)=1$ for all $n \geq 1$ and $u_2(2;x)=1$.  By \eqref{un(m;x)}, we have
\begin{align}
\widetilde{u}_n(x,y)&=\frac{1}{x}(\widetilde{u}_{n-1}(x,y)-\widetilde{u}_{n-1}(0,y))+\widetilde{u}_{n-1}(x,y)+\frac{xy}{1-x}(\widetilde{u}_{n-2}(0,y)-\widetilde{u}_{n-2}(0,xy))\notag\\
&\quad+y\widetilde{u}_{n-1}(0,y), \quad n \geq 3, \label{un(x,y)}
\end{align}
with $\widetilde{u}_1(x,y)=y$ and $\widetilde{u}_2(x,y)=y+y^2$. Note that \eqref{un(x,y)} is seen also to hold for $n=2$ if one defines $\widetilde{u}_0(x,y)=0$.   Multiplying both sides of \eqref{un(x,y)} by $t^n$, and summing over all $n \geq 2$, yields the functional equation
\begin{align}
u(t;x,y)&=yt+\frac{t}{x}(u(t;x,y)-u(t;0,y))+tu(t;x,y)
+\frac{xyt^2}{1-x}(u(t;0,y)-u(t;0,xy))+ytu(t;0,y). \label{2-21funeq}
\end{align}

We apply the kernel method to \eqref{2-21funeq} and let $x=t/(1-t)$ to obtain
\begin{align}
u(t;0,y)&=\frac{yt(1-2t)}{(1-t)(1-2t-yt+yt^2)}-\frac{yt^3}{(1-t)(1-2t-yt+yt^2)}u\left(t;0,\frac{yt}{1-t}\right). \label{u(t;0,y)}
\end{align}
Iteration of \eqref{u(t;0,y)}, where $|t|<a$ for some $a>0$ suitably small, then yields
\begin{align}
u(t;0,y)&=t(1-2t)\sum_{i\geq0}\frac{(-1)^iy^{i+1}t^{\frac{i(i+7)}{2}}}{\prod_{j=0}^i((1-2t)(1-t)^{j+1}-y(1-t)^2t^{j+1})}.\label{u(t;0,y)2}
\end{align}
By \eqref{2-21funeq}, we have
$$u(t;x,y)=\frac{xyt}{x-t-xt}-\frac{t((1-x)(1-xy)-x^2yt)}{(1-x)(x-t-xt)}u(t;0,y)-\frac{x^2yt^2}{(1-x)(x-t-xt)}u(t;0,xy),$$
and applying \eqref{u(t;0,y)2} gives
\small{\begin{align}
u(t;x,y)&=\frac{xyt}{x-t-xt}-\frac{t^2(1-2t)((1-x)(1-xy)-x^2yt)}{(1-x)(x-t-xt)}\sum_{i\geq0}\frac{(-1)^iy^{i+1}t^{\frac{i(i+7)}{2}}}{\prod_{j=0}^i((1-2t)(1-t)^{j+1}-y(1-t)^2t^{j+1})}\notag\\
&\quad-\frac{x^2yt^3(1-2t)}{(1-x)(x-t-xt)}\sum_{i\geq0}\frac{(-1)^i(xy)^{i+1}t^{\frac{i(i+7)}{2}}}{\prod_{j=0}^i((1-2t)(1-t)^{j+1}-xy(1-t)^2t^{j+1})}.\label{utxyform}
\end{align}}\normalsize

We wish to find $u(t;1,1)=\sum_{n\geq1}c_n(2\text{-}21)t^n$, though it is not possible to substitute $x=1$ directly into \eqref{utxyform} as it leads to the indeterminate $0/0$.  Note, however, for each $t$ of sufficiently small absolute value, there is uniform convergence in the $x$-derivatives of the partial sums of the series expansion in \eqref{utxyform} for all $x$ in some finite open interval containing $x=1$.  By, for example, \cite[p.\,187, Theorem 12]{Wade}, we thus may interchange the summation and differentiation in \eqref{utxyform} in applying  L'H\^{o}pital's rule to obtain
\begin{align*}
u(t;1,1)&=\lim_{x\rightarrow1}u(t;x,1)\\
&=\frac{t}{1-2t}-t^3\lim_{x\rightarrow1}\frac{\partial}{\partial x}\bigg(x^2\sum_{i\geq0}\frac{(-1)^it^{\frac{i(i+7)}{2}}}{\prod_{j=0}^i((1-2t)(1-t)^{j+1}-(1-t)^2t^{j+1})}\\
&\quad-x^2\sum_{i\geq0}\frac{(-1)^ix^{i+1}t^{\frac{i(i+7)}{2}}}{\prod_{j=0}^i((1-2t)(1-t)^{j+1}-x(1-t)^2t^{j+1})}\bigg)\\
&=\frac{t}{1-2t}+t^3\sum_{i\geq0}\frac{(-1)^it^{\frac{i(i+7)}{2}}\left(i+1+\sum_{j=0}^{i}\frac{(1-t)^2t^{j+1}}{(1-2t)(1-t)^{j+1}-(1-t)^2t^{j+1}}\right)}{\prod_{j=0}^i((1-2t)(1-t)^{j+1}-(1-t)^2t^{j+1})}\\
&=\frac{t}{1-2t}+t^3\sum_{i\geq0}\frac{(-1)^it^{\frac{i(i+7)}{2}}(1-2t)\sum_{j=0}^i\frac{(1-t)^{j+1}}{(1-2t)(1-t)^{j+1}-(1-t)^2t^{j+1}}}{\prod_{j=0}^i((1-2t)(1-t)^{j+1}-(1-t)^2t^{j+1})}.
\end{align*}
Replacing $i$ with $i-1$ (and then $j$ with $j-1$) in the last expression yields \eqref{th2-21e1}.
\end{proof}

We now consider avoidance of 3-21.  Applying similar reasoning to that used in the proof of Lemma \ref{2-21lem1}, we obtain the following recurrences for $v_n(m,a)$.

\begin{lemma}\label{3-21lem1}
If $n \geq 3$ and $2 \leq m \leq n$, then
\begin{align}
v_n(m,a)&=v_{n-1}(m,a-1)+v_{n-1}(m,a)+v_{n-1}(m,m), \quad 1 \leq a \leq m-1, \label{2-211rec1}\\
v_n(m,m)&=v_{n-1}(m-1,m-1)+v_{n-1}(m,m-1)+v_{n-1}(m,m), \label{2-211rec2}
\end{align}
with $v_n(1,1)=1$ for all $n \geq 1$ and $v_2(2,1)=0$, $v_2(2,2)=1$.
\end{lemma}

There is the following formula for the generating function of $c_n(3\text{-}21)$.

\begin{theorem}\label{th3-21}
We have
\begin{equation}\label{th3-21e1m}
\sum_{n\geq1}c_n(3\text{-}21)t^n=\frac{t}{1-2t}-\sum_{i\geq1}\frac{(-1)^it^{2i}(1-t)^{\binom{i}{2}}(1-2t)\sum_{j=1}^i\frac{t^{j}}{(1-3t+t^2)(1-t)^{j-1}-(1-2t)t^{j}}}{\prod_{j=1}^i\left((1-3t+t^2)(1-t)^{j-1}-(1-2t)t^{j}\right)}.
\end{equation}
\end{theorem}
\begin{proof}
Let $v_n(m;x)=\sum_{a=1}^mv_n(m,a)x^{m-a}$ for $n\geq m \geq 1$.  Then \eqref{2-211rec1} and \eqref{2-211rec2} together imply for $n \geq 3$ the recurrence
\begin{align*}
v_n(m;x)&=\frac{1}{x}(v_{n-1}(m;x)-v_{n-1}(m;0))+v_{n-1}(m;x)-v_{n-1}(m;0)
+\frac{1-x^m}{1-x}v_{n-1}(m;0)\\
&\quad+v_{n-1}(m-1;0), \quad 2 \leq m \leq n,
\end{align*}
with $v_n(1;x)=1$ for all $n \geq 1$ and $v_2(2;x)=1$.  This implies
\begin{align*}
\widetilde{v}_n(x,y)&=\frac{1}{x}(\widetilde{v}_{n-1}(x,y)-\widetilde{v}_{n-1}(0,y))+\widetilde{v}_{n-1}(x,y)-\widetilde{v}_{n-1}(0,y)
+\frac{1}{1-x}(\widetilde{v}_{n-1}(0,y)-\widetilde{v}_{n-1}(0,xy))\\
&\quad+y\widetilde{v}_{n-1}(0,y), \quad n \geq 2,
\end{align*}
with $\widetilde{v}_1(x,y)=y$.  Multiplying both sides of the last equality by $x^n$, and summing over $n \geq 2$, gives
\begin{align}
v(t;x,y)&=yt+\frac{t}{x}(v(t;x,y)-v(t;0,y))+tv(t;x,y)-tv(t;0,y)\notag\\
&\quad+\frac{t}{1-x}(v(t;0,y)-v(t;0,xy))+ytv(t;0,y), \label{fneqn3-21}
\end{align}
and taking $x=t/(1-t)$ in \eqref{fneqn3-21}, we obtain
\begin{align}
v(t;0,y)&=\frac{yt(1-2t)}{1-3t+t^2-yt(1-2t)}-\frac{t(1-t)}{1-3t+t^2-yt(1-2t)}v\left(t;0,\frac{yt}{1-t}\right). \label{v(t;0,y)}
\end{align}
Iteration of \eqref{v(t;0,y)} implies for $t$ sufficiently close to zero the explicit formula
\begin{align}
v(t;0,y)&=yt(1-2t)\sum_{i\geq0}\frac{(-1)^it^{2i}}{\prod_{j=0}^{i}(1-3t+t^2-\frac{yt^{j+1}(1-2t)}{(1-t)^j})}. \label{v(t,0,y)2}
\end{align}

Solving for $v(t;x,y)$ in \eqref{fneqn3-21} gives
\begin{align*}
v(t;x,y)&=\frac{xyt}{x-t-xt}-\frac{t((1-x)(1-xy)-x^2)}{(1-x)(x-t-xt)}v(t;0,y)
-\frac{xt}{(1-x)(x-t-xt)}v(t;0,xy),
\end{align*}
and applying \eqref{v(t,0,y)2} yields
\begin{align}
v(t;x,y)&=\frac{xyt}{x-t-tx}-\frac{yt^2(1-2t)((1-x)(1-xy)-x^2)}{(1-x)(x-t-xt)}\sum_{i\geq0}\frac{(-1)^it^{2i}}{\prod_{j=0}^{i}(1-3t+t^2-\frac{yt^{j+1}(1-2t)}{(1-t)^j})}\notag\\
&\quad-\frac{x^2yt^2(1-2t)}{(1-x)(x-t-xt)}\sum_{i\geq0}\frac{(-1)^it^{2i}}{\prod_{j=0}^{i}(1-3t+t^2-\frac{xyt^{j+1}(1-2t)}{(1-t)^j})}.\label{vtxyform}
\end{align}
As in the proof of \eqref{th2-21e1} above, we take the limit as $x\rightarrow1$ in $v(t;x,1)$ to obtain a formula for $v(t;1,1)=\sum_{n\geq1}c_n(3\text{-}21)t^n$ using \eqref{vtxyform}.  This gives
\begin{align*}
v(t;1,1)&=\lim_{x\rightarrow1}v(t;x,1)\\
&=\frac{t}{1-2t}-t^2\lim_{x\rightarrow1}\frac{\partial}{\partial x}\bigg(x^2\sum_{i\geq0}\frac{(-1)^it^{2i}}{\prod_{j=0}^i\left(1-3t+t^2-\frac{(1-2t)t^{j+1}}{(1-t)^j}\right)}\\
&\quad-x^2\sum_{i\geq0}\frac{(-1)^it^{2i}}{\prod_{j=0}^i\left(1-3t+t^2-\frac{x(1-2t)t^{j+1}}{(1-t)^j}\right)}\bigg)\\
&=\frac{t}{1-2t}+\sum_{i\geq0}\frac{(-1)^it^{2i+2}(1-t)^{\binom{i+1}{2}}(1-2t)\sum_{j=0}^i\frac{t^{j+1}}{(1-3t+t^2)(1-t)^j-(1-2t)t^{j+1}}}{\prod_{j=0}^i\left((1-3t+t^2)(1-t)^j-(1-2t)t^{j+1}\right)},
\end{align*}
and replacing $i$ with $i-1$ in the last expression yields \eqref{th3-21e1m}.
\end{proof}

\subsection{The cases 1-11 and 2-11}

To prove our formula in the case 1-11, we introduce and derive formulas for three auxiliary generating functions before proceeding.  Recall that a \emph{level} refers to an occurrence of equal adjacent entries within a sequence.  For the first generating function, we let $H_m=H_m(t)$ for $m \geq1$ enumerate the (nonempty) $m$-ary words $w_1\cdots w_n$ for all $n \geq1$ that contain no levels and end in $m$ such that $w_{i+1}\leq w_i+1$ for each $i \in[n-1]$, with $H_0=0$.

Let $U_n=U_n(t)$ denote the $n$-th Chebyshev polynomial of the second kind defined recursively by $U_n=2tU_{n-1}-U_{n-2}$ for $n \geq 2$, with initial values $U_0=1$ and $U_1=2t$ (see, e.g., \cite{Ri}). We have the following explicit formula for $H_m$ in terms of Chebyshev polynomials.

\begin{lemma}\label{1-11lem1}
If $m \geq1$, then
\begin{equation}\label{1-11lem1e1}
H_m(t)=\frac{U_{m-1}(s)}{U_{m+1}(s)},
\end{equation}
where $s=\frac{\sqrt{1+t}}{2\sqrt{t}}$.
\end{lemma}
\begin{proof}
Let $H_m^*=H_m^*(t)$ denote the restriction of $H_m$ to those words starting with $m$.  Note first that $H_m^*=t+t(H_m-H_m^*)$ for $m \geq1$, which follows from observing that a word $\rho$ enumerated by $H_m^*$ is of the form $\rho=m\rho'$, where either $\rho'$ is empty or is nonempty and starts with a letter in $[m-1]$.  Further, we also have
\begin{equation}\label{1-11lem1e2}
H_m=tH_{m-1}+H_m^*+tH_{m-1}(H_m-H_m^*), \qquad m \geq 2,
\end{equation}
with $H_1=1$.  To see \eqref{1-11lem1e2}, consider the following cases on the form of $w=w_1\cdots w_n$ enumerated by $H_m$ where $m \geq2$:  (i) $w=w'm$, where $w'$ is $(m-1)$-ary, (ii) $w$ starts with $m$ or (iii) $w=w'mw''m$, where $w'$ and $w''$ are nonempty and $w'$ does not contain $m$. It is seen that cases (i)--(iii) account for the three respective terms in \eqref{1-11lem1e2}, by subtraction as $w''$ cannot start with $m$ in (iii).  Substituting $H_m^*=\frac{t(1+H_m)}{1+t}$ into \eqref{1-11lem1e2}, and solving for $H_m$, yields the recurrence
\begin{equation}\label{1-11lem1e3}
H_m=\frac{t(1+H_{m-1})}{1-tH_{m-1}}, \qquad m \geq 2.
\end{equation}
Note that \eqref{1-11lem1e1} is seen also to hold for $m=1$, as $U_0(s)=1$ and $U_2(s)=\frac{1}{t}$.  Proceeding inductively using \eqref{1-11lem1e3}, we then have for $m \geq 2$,
$$H_m(t)=\frac{t(U_m(s)+U_{m-2}(s))}{U_m(s)-tU_{m-2}(s)}=\frac{\sqrt{t(1+t)}U_{m-1}(s)}{U_m(s)-tU_{m-2}(s)}=\frac{U_{m-1}(s)}{U_{m+1}(s)},$$
as desired, where we have used the respective facts $\sqrt{(1+t)/t}U_{m-1}(s)=U_m(s)+U_{m-2}(s)$ and $\sqrt{t(1+t)}U_{m+1}(s)=U_m(s)-tU_{m-2}(s)$ in the last two equalities.
\end{proof}

Let $J_m=J_m(t)$ for $m \geq1$ denote the generating function for the number of $m$-ary words $w_1\cdots w_n$ (possibly empty) containing no levels such that $w_{i+1}\leq w_i+1$ for each $i \in [n-1]$, with $J_0=1$.

\begin{lemma}\label{1-11lem2}
If $m \geq 1$, then
\begin{equation}\label{1-11lem2e1}
J_m(t)=\frac{1}{U_{m+1}(s)}\left(\frac{1+t}{t}\right)^{(m+1)/2}.
\end{equation}
\end{lemma}
\begin{proof}
Let $J_m^*=J_m^*(t)$ denote the restriction of $J_m$ to those (nonempty) words starting with $m$.  Then we have $J_m^*=t(J_m-J_m^*)$ for $m \geq 1$.  To realize this, note that a word $\rho$ enumerated by $J_m^*$ is of the form $\rho=m\rho'$, where $\rho'$ is either empty or is non-empty and starts with a letter in $[m-1]$, with $\rho'$ in the latter case being enumerated by $J_m-J_m^*-1$. We also have the recurrence
\begin{equation}\label{1-11lem2e2}
J_m=J_{m-1}+J_m^*+tH_{m-1}(J_m-J_m^*), \qquad m \geq 2,
\end{equation}
with $J_1=1+t$.  To show \eqref{1-11lem2e2}, note first that a word $w$ enumerated by $J_m$ for $m \geq 2$ either is $(m-1)$-ary, starts with $m$ or has the form $w=w'mw''$, where $w'$ and $w''$ are $(m-1)$- and $m$-ary, respectively.  The section $w'$ of $w$ must be nonempty, whereas $w''$ could be empty, and thus $w'$ and $w''$ are precisely of the form enumerated by $H_{m-1}$ and $J_m-J_m^*$, respectively. This accounts for the third term in \eqref{1-11lem2e2}, and implies the formula, where the extra $t$ factor accounts for the intermediate letter $m$ in the decomposition of $w$.

Substituting $J_m^*=\frac{t}{1+t}J_m$ into \eqref{1-11lem2e2}, and solving for $J_m$, yields
\begin{equation}\label{1-11lem2e3}
J_m=\frac{J_{m-1}}{1-\frac{t}{1+t}(1+H_{m-1})}, \qquad m \geq1,
\end{equation}
where the $m=1$ case follows from the initial values of $H_m$ and $J_m$.  By \eqref{1-11lem1e1} and \eqref{1-11lem2e3}, and the fact $U_m(s)+U_{m-2}(s)=\sqrt{(1+t)/t}U_{m-1}(s)$, we have
\begin{equation}\label{1-11lem2e4}
J_m(t)=\frac{J_{m-1}(t)}{1-\frac{t}{1+t}\left(1+\frac{U_{m-2}(s)}{U_m(s)}\right)}=\frac{U_m(s)J_{m-1}(t)}{U_m(s)-\frac{\sqrt{t}}{\sqrt{1+t}}U_{m-1}(s)}=\frac{U_m(s)J_{m-1}(t)}{\frac{\sqrt{t}}{\sqrt{1+t}}U_{m+1}(s)}, \quad m \geq 1,
\end{equation}
with $J_0(t)=1$.  Upon applying \eqref{1-11lem2e4} repeatedly, and observing a telescoping product, we obtain \eqref{1-11lem2e1}.
\end{proof}

Let $\mathcal{C}_n'$ denote the subset of $\mathcal{C}_n$ whose members contain no  levels (i.e., they avoid the 11 subword) and let $\mathcal{C}_{n,m}'$ be the subset of $\mathcal{C}_n'$ whose members have largest letter $m$ (i.e., that contain $m$ distinct letters).  Given $m \geq 2$ and $n \geq m-1$, let $\mathcal{C}_{n,m}''$ be the subset of $\mathcal{C}_{n,m-1}'$ whose members end in $m-1$.  Let $G_m=G_{m}(t)$ for $m \geq 2$ denote the generating function enumerating the members of $\mathcal{C}_{n,m}''$ for all $n \geq m-1$, with $G_1=1$.  We have the following explicit formula for $G_m$.

\begin{lemma}\label{1-11lem3}
If $m \geq1$, then
\begin{equation}\label{1-11lem3e1}
G_m(t)=\frac{1}{U_m(s)}\left(\frac{t}{1+t}\right)^{(m-2)/2}.
\end{equation}
\end{lemma}
\begin{proof}
The equality is clear for $m=1$, so assume $m \geq 2$.  Let $w \in \mathcal{C}_{n,m}''$ be decomposed as $w=w^{(1)}(m-1)\cdots w^{(k)}(m-1)$ for some $k \geq 1$, where each section $w^{(j)}$ does not contain $m-1$.  Note that $w \in \mathcal{C}_{n,m}''$ where $m\geq 3$ if and only if $w^{(1)}\in\mathcal{C}_{\ell,m-1}''$ for some $\ell \geq m-2$ and each $w^{(j)}$ for $2 \leq j \leq k$ is of the form enumerated by $H_{m-2}$.  Thus, by definitions of the generating functions and \eqref{1-11lem1e1}, we have
\begin{align}
G_m(t)&=\sum_{k\geq1}t^kG_{m-1}(t)(H_{m-2}(t))^{k-1}=\frac{tG_{m-1}(t)}{1-tH_{m-2}(t)}=\frac{tU_{m-1}(s)G_{m-1}(t)}{U_{m-1}(s)-tU_{m-3}(s)}\notag\\
&=\frac{\sqrt{t}U_{m-1}(s)G_{m-1}(t)}{\sqrt{1+t}U_{m}(s)}, \qquad m \geq2, \label{1-11lem3e2}
\end{align}
with $G_1(t)=1$, where in the last equality, we used the fact $\sqrt{t(1+t)}U_m(s)=U_{m-1}(s)-tU_{m-3}(s)$.  Iterating \eqref{1-11lem3e2}, and observing a telescoping product, yields \eqref{1-11lem3e1}.
\end{proof}

Now let $q_n(m)=|\mathcal{C}_{n,m}'|$ for $n \geq m \geq 1$.  Define $Q_m=Q_m(t)=\sum_{n \geq m}q_n(m)t^n$ for $m \geq1$.  Using the prior results, we may obtain a simple explicit formula for $Q_m$ in terms of Chebyshev polynomials.

\begin{lemma}\label{1-11lem4}
If $m \geq1$, then
\begin{equation}\label{1-11lem4e1}
Q_m(t)=\frac{\sqrt{1+t}}{\sqrt{t}U_m(s)U_{m+1}(s)}.
\end{equation}
\end{lemma}
\begin{proof}
We may assume $m \geq2$, the formula being clear for $m=1$ as $Q_1=t$. We decompose $w \in \mathcal{C}_{n,m}'$ as $w=w^{(1)}mw^{(2)}m\cdots w^{(k)}mw^{(k+1)}$ for some $k \geq 1$, where no $w^{(j)}$ contains $m$ or any levels and $w^{(j)}$ for each $j \in [k]$ is nonempty and ends in $m-1$.  By the definitions of the various generating functions, we obtain
$$Q_m=\frac{tG_mJ_{m-1}}{1-tH_{m-1}}, \qquad m \geq 2,$$
where the extra factor of $t$ in the numerator accounts for the $m$ directly following the section $w^{(1)}$.
Thus, by \eqref{1-11lem1e1}, \eqref{1-11lem2e1} and \eqref{1-11lem3e1}, we get
$$Q_m(t)=\frac{\frac{t}{U_m(s)}\left(\frac{t}{1+t}\right)^{(m-2)/2}\cdot\frac{1}{U_m(s)}\left(\frac{1+t}{t}\right)^{m/2}}{1-t\frac{U_{m-2}(s)}{U_m(s)}}=\frac{1+t}{U_m(s)(U_m(s)-tU_{m-2}(s))}=\frac{\sqrt{1+t}}{\sqrt{t}U_m(s)U_{m+1}(s)},$$
as desired.
\end{proof}

It is now possible to ascertain a formula for the generating function of $c_n(1\text{-}11)$.

\begin{theorem}\label{1-11th1}
We have
\begin{equation}\label{1-11th1e1}
\sum_{n\geq1}c_n(1\text{-}11)t^n=\sum_{m\geq1}\frac{(1+t)^{m+1}}{\sqrt{t(1+t)}U_m\left(\frac{\sqrt{1+t}}{2\sqrt{t}}\right)U_{m+1}\left(\frac{\sqrt{1+t}}{2\sqrt{t}}\right)}.
\end{equation}
\end{theorem}
\begin{proof}
First, note that members of $\mathcal{C}_n(1\text{-}11)$ can be obtained from those in $\mathcal{C}_{i,j}'$ for various $i$ and $j$ by choosing some subset of the distinct letters to have an initial run of length two.  That is, given $\sigma \in \mathcal{C}_{i,j}'$, we select $n-i$ of the $j$ distinct letters of $\sigma$ and for each chosen letter $x$, we insert a second $x$ just following the first occurrence of $x$ within $\sigma$.  It is seen that all members of $\mathcal{C}_n(1\text{-}11)$ arise uniquely in this manner as $i$ and $j$ vary and thus
$$c_n(1\text{-}11)=\sum_{i=1}^n\sum_{j=1}^n\binom{j}{n-i}q_i(j).$$
Note that the inner sum in this expression can only be nonzero if $i \leq n \leq 2i$, as $j \leq i$ is required, and hence the formula may be rewritten as
$$c_n(1\text{-}11)=\sum_{i=\lfloor\frac{n+1}{2}\rfloor}^n\sum_{j=n-i}^i\binom{j}{n-i}q_i(j).$$
Thus, we have
\begin{align*}
\sum_{n\geq1}c_n(1\text{-}11)t^n
&=\sum_{n\geq1}t^n\sum_{i=\lfloor\frac{n+1}{2}\rfloor}^n\sum_{j=n-i}^{i}\binom{j}{n-i}q_i(j)
=\sum_{i\geq1}\sum_{n=i}^{2i}t^n\sum_{j=n-i}^i\binom{j}{n-i}q_i(j)\\
&=\sum_{i\geq1}\sum_{j=1}^iq_i(j)\sum_{n=i}^{i+j}\binom{j}{n-i}t^n
=\sum_{i\geq1}\sum_{j=1}^it^i(1+t)^jq_i(j)\\
&=\sum_{j\geq1}\sum_{i\geq j}t^i(1+t)^jq_i(j)=\sum_{j\geq1}Q_j(t)(1+t)^j,
\end{align*}
from which \eqref{1-11th1e1} follows from the last equality and \eqref{1-11lem4e1}.
\end{proof}

Using some of the prior results, one can obtain with a bit more work a formula comparable to \eqref{1-11th1e1} for the pattern 2-11.

\begin{theorem}\label{2-11th}
We have
\begin{equation}\label{2-11the1}
\sum_{n\geq1}c_n(2\text{-}11)t^n=\sum_{m\geq1}\frac{\left(\frac{1+t}{t}\right)^{(m+1)/2}}{tU_m\left(\frac{\sqrt{1+t}}{2\sqrt{t}}\right)
U_{m+1}\left(\frac{\sqrt{1+t}}{2\sqrt{t}}\right)U_{m+2}\left(\frac{\sqrt{1+t}}{2\sqrt{t}}\right)}.
\end{equation}
\end{theorem}
\begin{proof}
Let $P_m=P_m(t)$ denote the generating function for the number of members $w \in \mathcal{C}_{n}(2\text{-}11)$ with $\text{max}(w)=m$.  Clearly, $P_1=\frac{t}{1-t}$, so assume $m \geq 2$.  To determine an explicit formula for $P_m$ where $m\geq 2$, first decompose $w$ enumerated by $P_m$ as $w=w^{(1)}m\cdots w^{(\ell)}mw^{(\ell+1)}$ for some $\ell \geq1$, where each $w^{(j)}$ is $(m-1)$-ary. Then $w^{(j)}$ for $j \in [\ell]$ ends in $m-1$, with $w^{(1)}$ a Catalan word that avoids 2-11 and $w_j$ for each $j \in [2,\ell+1]$ satisfying the Catalan adjacency requirement and containing no levels.  Further, we have that $w^{(1)}$ is nonempty as $m\geq2$, with all other $w^{(j)}$ possibly empty.
By the definitions of $H_m$ and $J_m$, we get
\begin{equation}\label{2-11the2}
P_m=\frac{tK_mJ_{m-1}}{1-t-tH_{m-1}}, \quad m \geq 1,
\end{equation}
where $K_m=K_m(t)$ for $m \geq2$ is the generating function for the number of members of $\mathcal{C}_{n}(2\text{-}11)$ in which $m-1$ is both the largest and last letter, with $K_1=1$.

To find a formula for $K_m$, suppose $\rho$ enumerated by $K_m$ where $m\geq2$ is decomposed as $\rho=\rho^{(1)}(m-1)\cdots\rho^{(\ell)}(m-1)$ for some $\ell \geq1$, where each $\rho^{(j)}$ is $(m-2)$-ary.  Then it is seen that the section $\rho^{(1)}(m-1)$ contributes $tK_{m-1}$ towards $K_m$, with the remaining part $\rho^{(2)}(m-1)\cdots\rho^{(\ell)}(m-1)$ contributing $\frac{1}{1-t-tH_{m-2}}$.  This implies
$$K_m=\frac{tK_{m-1}}{1-t-tH_{m-2}}, \qquad m \geq 2,$$
with $K_1=1$, and hence
$$K_m=\frac{t^{m-1}}{\prod_{i=0}^{m-2}(1-t-tH_i)}, \qquad m \geq1.$$
So, by \eqref{2-11the2}, we get
\begin{equation}\label{2-11the3}
P_m=\frac{t^mJ_{m-1}}{\prod_{i=0}^{m-1}(1-t-tH_i)}, \qquad m \geq 1.
\end{equation}
Note that by \eqref{1-11lem1e1} and the recurrence for $U_j$,
we have
$$1-t(1+H_i(t))=1-t\left(1+\frac{U_{i-1}(s)}{U_{i+1}(s)}\right)=\frac{U_{i+1}(s)-\sqrt{t(1+t)}U_i(s)}{U_{i+1}(s)}=\frac{tU_{i+3}(s)}{U_{i+1}(s)},$$
where $s=\frac{\sqrt{1+t}}{2\sqrt{t}}$. Hence, by \eqref{2-11the3} and \eqref{1-11lem2e1}, we get
\begin{align*}
P_m(t)&=\frac{t^mJ_{m-1}(t)}{\prod_{i=0}^{m-1}t\left(\frac{tU_{i+3}(s)}{U_{i+1}(s)}\right)}=\frac{U_1(s)U_2(s)J_{m-1}(t)}{U_{m+1}(s)U_{m+2}(s)}=\frac{U_1(s)U_2(s)\cdot\frac{1}{U_{m}(s)}\left(\frac{1+t}{t}\right)^{m/2}}{U_{m+1}(s)U_{m+2}(s)}\\
&=\frac{\left(\frac{1+t}{t}\right)^{(m+1)/2}}{tU_m(s)U_{m+1}(s)U_{m+2}(s)}.
\end{align*}
Summing the last formula over all $m\geq1$ now yields the desired result.
\end{proof}

\section{Patterns of the type $(2,1)$}

In this section, we seek to determine $c_n(\tau)$ for $\tau$ of the form $(2,1)$.  Our work is shortened by observing that the patterns 12-1, 12-2, 12-3 and 23-1 are equivalent on Catalan words to the analogous classical patterns of length three.  An argument similar to that given above showing the equivalence of 31-2 and 3-1-2 may be given in each case.  From the results in \cite{BKV2}, we obtain the following simple formulas for the number of avoiders of each pattern.

\begin{proposition}\label{(2,1)prop}
If $n \geq 1$, then $c_n(12\text{-}1)=c_n(12\text{-}3)=2^{n-1}$, $c_n(12\text{-}2)=\binom{n}{2}+1$ and $c_n(23\text{-}1)=F_{2n-1}$.
\end{proposition}

\noindent \emph{Remark:} The formulas from Propositions \ref{(1,2)prop} and \ref{(2,1)prop} were obtained in \cite{BKV2} as special cases of generating functions enumerating the various avoidance classes $\mathcal{C}_n(\tau)$, with $\tau$ a classical pattern of length three, according to the descents statistic.  We remark that it is possible to provide direct combinatorial proofs of these formulas for $c_n(\tau)$, where $\tau$ is taken to be the corresponding equivalent vincular pattern of length three in each case, without recourse to generating functions.  This can either be achieved bijectively by defining certain correspondences between discrete structures or by arguing combinatorially that the sequence $u_n=c_n(\tau)$ satisfies the appropriate two-term recurrence (for example, $u_n=3u_{n-1}-u_{n-2}$ for $n \geq 3$ in the cases when $\tau=2\text{-}12, 23\text{-}1$ or $u_n=4u_{n-1}-3u_{n-2}$ when $\tau=3\text{-}12$).  We leave these problems as exercises to explore for the interested reader. \medskip

 The pattern 11-1 is equivalent to 1-11 considered above.

\begin{proposition}\label{11-1prop}
If $n \geq 1$, then $c_n(11\text{-}1)=c_n(1\text{-}11)$.  Moreover, the refinements of $c_n(11\text{-}1)$ according to the largest or last letter parameters coincide with those of $c_n(1\text{-}11)$.
\end{proposition}
\begin{proof}
First note that $\pi\in\mathcal{C}_n$ belongs to $\mathcal{C}_n(11\text{-}1)$ if and only if for each $i \geq 1$, all runs of $i$ within $\pi$ are of length one, except for possibly the last, which may be of length two.  Likewise, membership of $\pi$ in $\mathcal{C}_n(1\text{-}11)$ implies each run of $i$ is of length one, except for possibly the first, which may be of length two.  Let $\pi'$ be obtained from $\pi\in \mathcal{C}_n(11\text{-}1)$ by moving the second letter, if it occurs, in the last run of $i$ to the first run, leaving all other runs of $i$ unchanged, for each of the distinct letters $i$ occurring in $\pi$.  The mapping $\pi \mapsto \pi'$ is seen to be a bijection from $\mathcal{C}_n(11\text{-}1)$ to $\mathcal{C}_n(1\text{-}11)$.  Furthermore, since $\pi$ and $\pi'$ have the same largest and last letters for all $\pi$, the second statement follows.
\end{proof}

\noindent \emph{Remark:} Similarly, by reversing the sequence of run lengths of each letter and considering the resulting Catalan word, one has, more generally, $c_n(1^{a_1}\text{-}\cdots\text{-}1^{a_k})=c_n(1^{a_k}\text{-}\cdots\text{-}1^{a_1})$, where $k \geq 2$ and $a_1,\ldots,a_k\geq1$. \medskip

\subsection{The cases 22-1 and 32-1}

We have the following simple formulas for the number of avoiders of 22-1 or 32-1.

\begin{theorem}\label{22-1/32-1}
If $n \geq 1$, then $c_n(22\text{-}1)=L_n$ and $c_n(32\text{-}1)=\frac{1}{2}\left(3^{n-1}+1\right)$.
\end{theorem}
\begin{proof}
Let $a_n=c_n(22\text{-}1)$.  Recall that there are $M_{n-1}$ members of $\mathcal{C}_n$ that have no equal adjacent letters, i.e., contain no levels; see, e.g., \cite{BKRV}.  Consider whether or not $\pi=\pi_1\cdots\pi_n \in \mathcal{C}_n(22\text{-}1)$ contains a level, and if it does, let $i$ denote the index of the first letter of the leftmost level. In the latter case, we have that $\pi$ may be decomposed as $\pi=\alpha\beta$, where $\alpha \in \mathcal{C}_i$ contains no levels (and hence is enumerated by $M_{i-1}$) and $\beta=\pi_{i+1}\cdots \pi_n$ is such that subtracting $x-1$ from each of its entries results in an arbitrary member of $\mathcal{C}_{n-i}(22\text{-}1)$, where $x$ denotes the common value of $\pi_i$ and $\pi_{i+1}$.  Note that $\beta$ cannot contain any letters in $[x-1]$ and is as described, for otherwise $\pi$ would contain an occurrence of 22-1.  Combining the cases for $\pi$, we get the recurrence
\begin{equation}\label{a_nrecur}
a_n=M_{n-1}+\sum_{i=1}^{n-1}M_{i-1}a_{n-i}, \qquad n \geq 2,
\end{equation}
with $a_1=1$. On the other hand, considering whether or not a Motzkin left-factor with $n-1$ steps ends at height zero, and if not, the position of the rightmost step starting at height zero, implies that the sequence $L_n$ also satisfies \eqref{a_nrecur} for $n \geq 2$, with $L_1=1$, which yields the first equality.

For the second equality, let $b_n=c_n(32\text{-}1)$ and first observe that there are $b_{n-1}$ members of $\mathcal{C}_n(32\text{-}1)$ containing a single 1 and the same number that start $1, 1$.  So assume $\rho=\rho_1\cdots \rho_n \in \mathcal{C}_n(32\text{-}1)$ is of the form $\rho=1\rho'1\rho''$, where $\rho'$ is nonempty and contains no $1$'s.  If $\rho'$ has length $m-1$, where $2 \leq m \leq n-1$, then there are $2^{m-2}$ possibilities for $\rho'$ as it is weakly increasing and starts with 2.  Further, there are $b_{n-m}$ possibilities for the section $1\rho''$, as no restrictions are placed upon it by $1\rho'$.  Considering all possible $m$ then yields the recurrence
\begin{equation}\label{b_nrecur}
b_n=2b_{n-1}+\sum_{m=2}^{n-1}2^{m-2}b_{n-m}, \qquad n \geq 2,
\end{equation}
with $b_1=1$.

At this point, one may show $b_n=\frac{1}{2}(3^{n-1}+1)$ for $n \geq1$ by computing its generating function using \eqref{b_nrecur}.  Alternatively, one can give a more combinatorial proof as follows.  First, recall that there are $\frac{1}{2}(3^{n-1}+1)$ partitions of $[n]$ that have at most three blocks; see, for example, \cite[A124302]{Sloane}.  Let us represent partitions sequentially as restricted growth words (see, e.g., \cite{Wag}) and let $\mathcal{R}_n$ denote the set of such words corresponding to the partitions of $[n]$ having at most three blocks.  Let $r_n=|\mathcal{R}_n|$ for $n \geq 1$ and we show that $r_n$ satisfies \eqref{b_nrecur} for all $n\geq2$.  Note that members of $\mathcal{R}_n$ are $3$-ary words of length $n$ that  start with $1$, such that if 3 occurs, then 2 does as well and the leftmost 3 occurs somewhere to the right of the leftmost 2.  Then there are clearly $r_{n-1}$ members of $\mathcal{R}_n$ that end in $1$ and the same number ending in 2.  To enumerate the members ending in 3, let $n-m+1$ be the position of the rightmost 2
within $\pi=\pi\cdots\pi_n \in \mathcal{R}_n$, where $2\leq m\leq n-1$.  Then there are $r_{n-m}$ possibilities for the section $\pi_1\cdots\pi_{n-m}$, and $2^{m-2}$ possibilities for $\pi_{n-m+2}\cdots\pi_{n-1}$, each entry of which must be a 1 or 3 independent of the others.  Summing over $2 \leq m\leq n-1$ then gives $\sum_{m=2}^{n-1}2^{m-2}r_{n-m}$ members of $\mathcal{R}_n$ ending in 3.  Combining with the prior two cases implies $r_n$ satisfies recurrence \eqref{b_nrecur}, with $r_1=1$.  Hence, we have $b_n=r_n$ for all $n \geq 1$, as desired.
\end{proof}

\subsection{The cases 21-2 and 21-3}

Given $n \geq 1$ and $1 \leq a \leq n$, let $\mathcal{U}_{n,a}$ denote the subset of $\mathcal{C}_n(21\text{-}2)$ whose members have last letter $a$ and let $u_n(a)=|\mathcal{U}_{n,a}|$.  Put $u_n(a)=0$ if it is not the case that $n \geq 1$ with $a \in [n]$.  The array $u_n(a)$ is given recursively as follows.

\begin{lemma}\label{21-2lem1}
If $n \geq 2$, then
\begin{equation}\label{21-2lem1e1}
u_n(a)=\sum_{j=a-1}^{n-1}u_{n-1}(j)-\sum_{j=1}^{a-1}\sum_{m=j+1}^{n-a}\binom{m-2}{j-1}u_{n-m}(a), \quad 2 \leq a \leq n,
\end{equation}
with $u_n(1)=\sum_{j=1}^{n-1}u_{n-1}(j)$ for $n \geq 2$ and $u_1(1)=1$.
\end{lemma}
\begin{proof}
The initial condition when $a=1$ is clear, upon appending $1$ to an arbitrary member of $\mathcal{C}_n(21\text{-}2)$, which is seen not to introduce an occurrence of 21-2.  Note that the set $\mathcal{U}_{n,n}$ for $n \geq 1$ consists of only the sequence $12\cdots n$, and thus \eqref{21-2lem1e1} is seen to hold for $a=n$.  So assume $\pi=\pi_1\cdots \pi_n \in \mathcal{U}_{n,a}$, where $n \geq 3$ and $2 \leq a \leq n-1$.  By a \emph{descent top} within $\pi$, it is meant the larger letter in a descent, i.e., the letter $\pi_i$ within an adjacency $\pi_i\pi_{i+1}$ for some $1 \leq i \leq n-1$ such that $\pi_i>\pi_{i+1}$.
Let $\pi'=\pi_{i_1}\cdots \pi_{i_r}$ denote the subsequence of $\pi$ for some $r \geq 1$ wherein $i_r=n$ and $\pi_{i_j}$ for $1 \leq j \leq r-1$ is the $j$-th descent top of $\pi$ (from the left).  One may verify that $\pi \in \mathcal{C}_n(21\text{-}2)$ implies $\pi'$ is strictly decreasing (with the converse holding as well).  This implies one may append an $a$ to any member of $\cup_{j=a}^{n-1}\mathcal{U}_{n-1,j}$ without introducing 21-2, with all members of $\mathcal{U}_{n,a}$ not ending in $a-1, a$ arising uniquely in this manner. On the other hand, we may append $a$ to $\rho \in \mathcal{U}_{n-1,a-1}$ if and only if $a$ does not occur as a descent top in $\rho$.

We proceed by enumerating the subset $S$ of $\mathcal{U}_{n-1,a-1}$ whose members contain $a$ as a descent top and subtracting the result from $u_{n-1,a-1}$.  Note that $S$ is empty if $a=n-1$, so we may assume $2 \leq a \leq n-2$.  Thus, to complete the proof, it suffices to show
\begin{equation}\label{21-2lem1e2}
|S|=\sum_{j=1}^{a-1}\sum_{m=j+1}^{n-a}\binom{m-2}{j-1}u_{n-m}(a), \quad 2 \leq a \leq n-2.
\end{equation}
To do so, let $\tau=\tau_1\cdots \tau_n \in S$ and suppose $\tau_{n-m}\tau_{n-m+1}=a(a-j)$ for some $1 \leq j \leq a-1$ and $m$.  We seek to enumerate such $\tau$.  Note first that since $\tau$ is a Catalan word, we have that $\tau_{n-m}=a$ implies $m \leq n-a$ and $\tau_{n-m+1}=a-j, \tau_{n-1}=a-1$ implies $m \geq j+1$, as each member of $[a-j,a-1]$ must occur at least once in the weakly increasing subsequence $\beta=\tau_{n-m+1}\cdots \tau_{n-1}$.  Thus, each such $\tau$ is of the form $\tau=\alpha\beta$, where $\alpha \in \mathcal{U}_{n-m,a}$ and $\beta$ is as given.  There are $u_{n-m}(a)$ possibilities for $\alpha$ and
$\binom{m-j-1+j-1}{j-1}=\binom{m-2}{j-1}$ for $\beta$, as the latter are synonymous with weak compositions of $m-j-1$ with $j$ parts.  Hence, there are $\binom{m-2}{j-1}u_{n-m}(a)$ such $\tau$ of the specified form and summing over all possible $j$ and $m$ yields \eqref{21-2lem1e2}, which completes the proof of \eqref{21-2lem1e1}.
\end{proof}

Define the generating function
$$u(t;x)=\sum_{n\geq1}\left(\sum_{a=1}^nu_n(a)x^{a-1}\right)t^n.$$
Then $u(t;x)$ satisfies the following functional equation.

\begin{lemma}\label{21-2lem2}
We have
\begin{equation}\label{21-2lem2e1}
\left(1+\frac{x^2t}{1-x}+\frac{t^2}{1-2t}\right)u(t;x)=t+\frac{t}{1-x}u(t;1)+\frac{t^2}{1-2t}u\left(t;\frac{xt}{1-t}\right).
\end{equation}
\end{lemma}
\begin{proof}
First note
\begin{align*}
&\sum_{n\geq2}\sum_{a=2}^n\sum_{j=a-1}^{n-1}u_{n-1}(j)x^{a-1}t^n=\sum_{n\geq2}\sum_{j=1}^{n-1}\sum_{a=2}^{j+1}u_{n-1}(j)x^{a-1}t^n=\sum_{n\geq2}\sum_{j=1}^{n-1}u_{n-1}(j)t^n\cdot \frac{x-x^{j+1}}{1-x}\\
&=\frac{xt}{1-x}\sum_{n\geq1}\sum_{j=1}^nu_n(j)(1-x^j)t^n=\frac{xt}{1-x}\left(u(t;1)-xu(t;x)\right),
\end{align*}
and hence
\begin{align*}
&\sum_{n\geq2}\sum_{j=1}^{n-1}u_{n-1}(j)t^n+\sum_{n\geq2}\sum_{a=2}^n\sum_{j=a-1}^{n-1}u_{n-1}(j)x^{a-1}t^n=tu(t;1)+\frac{xt}{1-x}\left(u(t;1)-xu(t;x)\right)\\
&=\frac{t}{1-x}\left(u(t;1)-x^2u(t;x)\right).
\end{align*}
Also, we have
\begin{align*}
&\sum_{n\geq 4}\sum_{a=2}^{n-2}\sum_{j=1}^{a-1}\sum_{m=j+1}^{n-a}\binom{m-2}{j-1}u_{n-m}(a)x^{a-1}t^n=\sum_{a \geq 2}\sum_{j=1}^{a-1}\sum_{n \geq a+2}\sum_{m=j+1}^{n-a}\binom{m-2}{j-1}u_{n-m}(a)x^{a-1}t^n\\
&=\sum_{a\geq 2}x^{a-1}\sum_{j=1}^{a-1}\sum_{m \geq j+1}\binom{m-2}{j-1}\sum_{n \geq a+m}u_{n-m}(a)t^n=\sum_{a\geq2}x^{a-1}\sum_{j=1}^{a-1}\sum_{n \geq a}u_n(a)t^n\sum_{m \geq j+1}\binom{m-2}{j-1}t^m\\
&=\sum_{a \geq 2}x^{a-1}\sum_{n \geq a}u_n(a)t^n\sum_{j=1}^{a-1}\frac{t^{j+1}}{(1-t)^j}=\frac{t(1-t)}{1-2t}\sum_{a \geq 2}x^{a-1}\sum_{n \geq a}u_n(a)t^n\left(\frac{t}{1-t}-\left(\frac{t}{1-t}\right)^a\right)\\
&=\frac{t^2}{1-2t}\sum_{n\geq1}\sum_{a=1}^nu_n(a)\left(x^{a-1}-\left(\frac{xt}{1-t}\right)^{a-1}\right)t^n=\frac{t^2}{1-2t}\left(u(t;x)-u\left(t;\frac{xt}{1-x}\right)\right).
\end{align*}
Thus, multiplying both sides of \eqref{21-2lem1e1} by $x^{a-1}t^n$, summing over all $n \geq 2$ and $2 \leq a \leq n$, and combining the result with $\sum_{n \geq 2}u_n(1)t^n$ yields
$$u(t;x)-t=\frac{t}{1-x}\left(u(t;1)-x^2u(t;x)\right)-\frac{t^2}{1-2t}\left(u(t;x)-u\left(t;\frac{xt}{1-x}\right)\right),$$
which rearranges to give \eqref{21-2lem2e1}.
\end{proof}

There is the following explicit formula for the generating function of 21-2.

\begin{theorem}\label{21-2thm}
We have
\begin{align}\label{21-2thme1}
\sum_{n\geq1}c_n(21\text{-}2)t^n&=\frac{t+\sum_{j\geq1}\frac{t^{2j+1}(1-t)^{\binom{j-1}{2}-1}\prod_{i=1}^j((1-2t)(1-t)^i-(1-t)t^i)}{\prod_{i=1}^j((1-2t)(1-t)^{2i}-t^i(1-t)^{i+1}+t^{2i+1})}}{1-2t-\sum_{j\geq1}\frac{t^{2j+1}(1-2t)(1-t)^{\binom{j}{2}-1}\prod_{i=1}^j((1-2t)(1-t)^i-(1-t)t^i)}{((1-2t)(1-t)^{j-1}-t^j)\prod_{i=1}^j((1-2t)(1-t)^{2i}-t^i(1-t)^{i+1}+t^{2i+1})}}.
\end{align}
\end{theorem}
\begin{proof}
We seek an expression for $u(t;1)$, which equals $\sum_{n\geq1}c_n(21\text{-}2)t^n$, by the definitions. First note that \eqref{21-2lem2e1} can be written as
\begin{equation}\label{21-2funeq1}
k(t;x)u(t;x)=a(t;x)+b(t)u\left(t;\frac{xt}{1-t}\right),
\end{equation}
where $k(t;x)=1+\frac{x^2t}{1-x}+\frac{t^2}{1-2t}$, $a(t;x)=t+\frac{t}{1-x}u(t;1)$ and $b(t)=\frac{t^2}{1-2t}$.
Taking $x=x(t)=\frac{1-t}{1-2t}$ in \eqref{21-2funeq1} cancels out the left-hand side and yields
\begin{equation}\label{21-2funeq2}
u(t;1)=x(t)-1+\frac{t(x(t)-1)}{1-2t}u\left(t;\frac{tx(t)}{1-t}\right).
\end{equation}
Iterating \eqref{21-2funeq1} for $t$ sufficiently close to zero gives
\begin{equation}\label{21-2funeq3}
u(t;x)=\sum_{j\geq0}
\frac{a\left(t;\frac{xt^j}{(1-t)^j}\right)}{k\left(t;\frac{xt^j}{(1-t)^j}\right)}\prod_{i=0}^{j-1}\frac{b(t)}{k\left(t;\frac{xt^i}{(1-t)^i}\right)}.
\end{equation}
Replacing $x$ with $\frac{x(t)t}{1-t}$ in \eqref{21-2funeq3}, we get
$$u\left(t;\frac{x(t)t}{1-t}\right)=\sum_{j\geq1}
\frac{a\left(t;\frac{x(t)t^j}{(1-t)^j}\right)}{k\left(t;\frac{x(t)t^j}{(1-t)^j}\right)}
\prod_{i=1}^{j-1}\frac{b(t)}{k\left(t;\frac{x(t)t^i}{(1-t)^i}\right)},$$
which is equivalent to
\begin{equation}\label{21-2funeq4}
u\left(t;\frac{x(t)t}{1-t}\right)=\sum_{j\geq1}
\frac{t^{2j-2}a\left(t;\frac{x(t)t^{j}}{(1-t)^j}\right)}
{(1-2t)^{j-1}\prod_{i=1}^jk\left(t;\frac{x(t)t^i}{(1-t)^i}\right)}.
\end{equation}

By \eqref{21-2funeq2}, \eqref{21-2funeq4} and the definition of $a(t;x)$, we have
\begin{align*}
u(t;1)&=x(t)-1+(x(t)-1)u(t;1)\sum_{j\geq1}
\frac{t^{2j}}
{(1-2t)^j\left(1-\frac{x(t)t^j}{(1-t)^j}\right)\prod_{i=1}^jk\left(t;\frac{x(t)t^i}{(1-t)^i}\right)}\\
&\quad+(x(t)-1)\sum_{j\geq1}
\frac{t^{2j}}
{(1-2t)^{j}\prod_{i=1}^jk\left(t;\frac{x(t)t^i}{(1-t)^i}\right)},
\end{align*}
which, upon solving for $u(t;1)$, implies
\begin{align*}
u(t;1)&=\frac{(x(t)-1)\left(1+\sum_{j\geq1}\frac{t^{2j}}{(1-2t)^j\prod_{i=1}^jk\left(t;\frac{x(t)t^i}{(1-t)^i}\right)}\right)}{1+(1-x(t))\sum_{j\geq1}\frac{t^{2j}}{(1-2t)^j\left(1-\frac{x(t)t^j}{(1-t)^j}\right)\prod_{i=1}^jk\left(t;\frac{x(t)t^i}{(1-t)^i}\right)}}.
\end{align*}
Note that for each $i \geq 1$,
\begin{align*}
&k\left(t;\frac{x(t)t^i}{(1-t)^i}\right)=1+\frac{x^2(t)t^{2i+1}}{(1-t)^i((1-t)^i-x(t)t^i)}+\frac{t^2}{1-2t}=\frac{(1-t)^2}{1-2t}+\frac{\frac{(1-t)^2}{1-2t}t^{2i}(x(t)-1)}{(1-t)^i((1-t)^i-x(t)t^i)}\\
&=\frac{(1-t)^2}{1-2t}\left(1+\frac{t^{2i}(x(t)-1)}{(1-t)^i((1-t)^i-x(t)t^i)}\right)=\frac{(1-t)^2}{1-2t}\left(1+\frac{t^{2i+1}}{(1-2t)(1-t)^{2i}-t^i(1-t)^{i+1}}\right)\\
&=\frac{(1-2t)(1-t)^{2i}-t^i(1-t)^{i+1}+t^{2i+1}}{(1-t)^{i-2}(1-2t)((1-2t)(1-t)^i-(1-t)t^i)},
\end{align*}
where we have made use of the fact $x^2(t)=\frac{(1-t)^2(x(t)-1)}{t(1-2t)}$ in the second equality.  Substituting this into the last formula above for $u(t;1)$, and simplifying, yields \eqref{21-2thme1}.
\end{proof}

Let $v_n(a)$ and $v(t;x)$ be defined analogously for the pattern 21-3 as $u_n(a)$ and $u(t;x)$ were above for 21-2.

\begin{lemma}\label{21-3lem1}
If $n \geq 2$, then
\begin{equation}\label{21-3lem1e1}
v_n(a)=\sum_{j=a-1}^{n-1}v_{n-1}(j)-\sum_{j=2}^{a-1}\sum_{m=j+1}^{n-a+1}\binom{m-2}{j-1}v_{n-m}(a-1), \quad 2 \leq a \leq n,
\end{equation}
with $v_n(1)=\sum_{j=1}^{n-1}v_{n-1}(j)$ for $n \geq 2$ and $v_1(1)=1$.
\end{lemma}
\begin{proof}
A subtraction argument similar to that given above for \eqref{21-2lem1e1} applies here which we briefly describe. Given $\pi=\pi_1\cdots\pi_n\in \mathcal{C}_n$, let $\pi'$ denote as before the subsequence consisting of the descent tops of $\pi$ together with the last letter.  Then we have that $\pi$ avoids 21-3 if and only if $\pi'$ is weakly decreasing.  Let $\mathcal{V}_{n,a}$ denote the subset $\mathcal{C}_n(21\text{-}3)$ whose members end in $a$.  Then one may append $a$ to any member of $\cup_{j=a-1}^{n-1}\mathcal{V}_{n-1,j}$ except for those belonging to $\mathcal{V}_{n-1,a-1}$ in which $a-1$ also occurs as a descent top.  Let $\rho$ denote such a member of $\mathcal{V}_{n-1,a-1}$, which may be decomposed as $\rho=\rho'(a-1)\rho''(a-1)$, where the first $a-1$ indicated corresponds to the rightmost descent top and $\rho''$ has first letter $a-j$ with $|\rho''|=m-2$.  The assumptions on $\rho$ imply  the restrictions $n \geq a+2$, $a \geq 3$, $j \in [2,a-1]$ and $m \in [j+1,n-a+1]$, with $\rho''$ weakly increasing.  Hence, we have that there are $\binom{m-2}{j-1}v_{n-m}(a-1)$ possibilities for $\rho$ for each $j$ and $m$.  Considering all $j$ and $m$ then gives the cardinality of the $\rho$ that must be excluded from $\cup_{j=a-1}^{n-1}\mathcal{V}_{n-1,j}$ when appending $a$, which accounts for the subtracted quantity in \eqref{21-3lem1e1}.
\end{proof}

Proceeding as before, one can show that the recurrence \eqref{21-3lem1e1} can be rewritten in terms of generating functions as follows.

\begin{lemma}\label{21-3lem2}
We have
\begin{equation}\label{21-3lem2e1}
\left(1+\frac{x^2t}{1-x}+\frac{xt^3}{(1-t)(1-2t)}\right)v(t;x)=t+\frac{t}{1-x}v(t;1)+\frac{xt^3}{(1-t)(1-2t)}v\left(t;\frac{xt}{1-t}\right).
\end{equation}
\end{lemma}

\begin{theorem}\label{21-3thm}
We have
\begin{equation}\label{21-3thme1}
\sum_{n\geq1}c_n(21\text{-}3)t^n=\frac{(x(t)-1)\left(1+\sum_{j\geq1}\frac{t^{\binom{j}{2}+3j}x(t)^j}{(1-t)^{\binom{j+1}{2}}(1-2t)^j\prod_{i=1}^jk\left(t;\frac{x(t)t^i}{(1-t)^i}\right)}\right)}{1+(1-x(t))\sum_{j\geq1}\frac{t^{\binom{j}{2}+3j}x(t)^j}{(1-t)^{\binom{j}{2}}(1-2t)^j((1-t)^j-x(t)t^j)\prod_{i=1}^jk\left(t;\frac{x(t)t^i}{(1-t)^i}\right)}},
\end{equation}
where $x(t)=\frac{1-3t+2t^2-t^3-\sqrt{1-10t+37t^2-62t^3+46t^4-12t^5+t^6}}{2t(1-3x+t^2)}$ and $k(t;x)=1+\frac{x^2t}{1-x}+\frac{xt^3}{(1-t)(1-2t)}$.
\end{theorem}
\begin{proof}
First note that \eqref{21-3lem2e1} may be written as
\begin{equation}\label{21-3the1}
k(t;x)v(t;x)=a(t;x)+b(t;x)v\left(t;\frac{xt}{1-t}\right),
\end{equation}
where $a(t;x)=t+\frac{t}{1-x}v(t;1)$, $b(t;x)=\frac{xt^3}{(1-t)(1-2t)}$ and $k(t;x)$ is as defined.  Taking $x=x(t)$ in \eqref{21-3the1}, where $x(t)$ is as given, is seen to cancel out the left-hand side and implies
\begin{equation}\label{2-13the2}
v(t;1)=x(t)-1+\frac{t^2x(t)(x(t)-1)}{(1-t)(1-2t)}v\left(t;\frac{x(t)t}{1-t}\right).
\end{equation}
Iterating \eqref{21-3the1} gives
$$v(t;x)=\sum_{j\geq0}
\frac{a\left(t;\frac{xt^j}{(1-t)^j}\right)}{k\left(t;\frac{xt^j}{(1-t)^j}\right)}\prod_{i=0}^{j-1}\frac{b\left(t;\frac{xt^i}{(1-t)^i}\right)}{k\left(t;\frac{xt^i}{(1-t)^i}\right)},$$
and hence
\begin{align*}
v\left(t;\frac{x(t)t}{1-t}\right)&=\sum_{j\geq1}\frac{a\left(t;\frac{x(t)t^j}{(1-t)^j}\right)}{k\left(t;\frac{x(t)t^j}{(1-t)^j}\right)}\prod_{i=1}^{j-1}\frac{b\left(t;\frac{x(t)t^i}{(1-t)^i}\right)}{k\left(t;\frac{x(t)t^i}{(1-t)^i}\right)}=\sum_{j\geq1}\frac{t^{3j-3}a\left(t;\frac{x(t)t^j}{(1-t)^j}\right)\prod_{i=1}^{j-1}\frac{x(t)t^i}{(1-t)^i}}{(1-3t+2t^2)^{j-1}\prod_{i=1}^jk\left(t;\frac{x(t)t^i}{(1-t)^i}\right)}\\
&=\sum_{j\geq1}\frac{t^{\binom{j}{2}+3j-3}x(t)^{j-1}a\left(t;\frac{x(t)t^j}{(1-t)^j}\right)}{(1-t)^{\binom{j}{2}}(1-3t+2t^2)^{j-1}\prod_{i=1}^jk\left(t;\frac{x(t)t^i}{(1-t)^i}\right)}.
\end{align*}
By \eqref{2-13the2} and the definition of $a(t;x)$, we then get
\begin{align*}
v(t;1)&=x(t)-1+\frac{t^2x(t)(x(t)-1)}{(1-t)(1-2t)}\sum_{j\geq1}\frac{t^{\binom{j}{2}+3j-3}x(t)^{j-1}\left(t+\frac{tv(t;1)}{1-\frac{x(t)t^j}{(1-t)^j}}\right)}{(1-t)^{\binom{j}{2}}(1-3t+2t^2)^{j-1}\prod_{i=1}^jk\left(t;\frac{x(t)t^i}{(1-t)^i}\right)}\\
&=x(t)-1+(x(t)-1)\sum_{j\geq1}\frac{t^{\binom{j}{2}+3j}x(t)^{j}}{(1-t)^{\binom{j+1}{2}}(1-2t)^{j}\prod_{i=1}^jk\left(t;\frac{x(t)t^i}{(1-t)^i}\right)}\\
&\quad+(x(t)-1)v(t;1)\sum_{j\geq1}\frac{t^{\binom{j}{2}+3j}x(t)^{j}}{(1-t)^{\binom{j}{2}}(1-2t)^{j}((1-t)^j-x(t)t^j)\prod_{i=1}^jk\left(t;\frac{x(t)t^i}{(1-t)^i}\right)}.
\end{align*}
Solving for $v(t;1)$ in the last equality yields \eqref{21-3thme1}.
\end{proof}

\subsection{The pattern 31-2}

In order to determine a generating function formula for $c_n(31\text{-}2)$, we refine the number of 31-2 avoiders according to a pair of parameters as follows.  Given $n \geq 2$ and $1\leq b \leq a \leq n-1$, let $w_n(a,b)$ denote the number of members of $\mathcal{C}_n(31\text{-}2)$ that contain exactly $a$ 1's and $b$ runs of $1$'s.  Note that
\begin{equation}\label{31-2key}
c_n(31\text{-}2)=1+\sum_{a=1}^{n-1}\sum_{b=1}^aw_n(a,b), \qquad n \geq1,
\end{equation}
by the definitions, where the added one accounts for the all 1's Catalan sequence.  The array $w_n(a,b)$
is given recursively as follows.

\begin{lemma}\label{31-2lem1}
If $n \geq 3$ and $1 \leq b \leq a \leq n-1$, then
\begin{equation}\label{31-2lem1e1}
w_n(a,b)=\binom{a-1}{b-1}\binom{n-a}{b-1}+\sum_{c=1}^{n-a-1}\sum_{d=1}^cw_{n-a}(c,d)\binom{a-1}{b-1}\binom{c-d+1}{b-1},
\end{equation}
with $w_2(1,1)=1$.
\end{lemma}
\begin{proof}
The initial condition when $n=2$ is clear, so assume $n \geq3$.  Let $\mathcal{W}_{n,a,b}$ denote the subset of $\mathcal{C}_n(31\text{-}2)$ enumerated by $w_n(a,b)$.  To show \eqref{31-2lem1e1}, we first enumerate the members $\pi \in \mathcal{W}_{n,a,b}$ consisting of only $1$'s and $2$'s (i.e., the binary members).  Note that there are $\binom{a-1}{b-1}$ ways of arranging $a$ $1$'s in $b$ runs within such $\pi$.  Once this has been done, we then place $n-a$ $2$'s in either $b$ or $b-1$ runs, the arrangements of which are synonymous with the weak compositions of $n-a-b+1$ with $b$ parts.  Then there are $\binom{(n-a-b+1)+b-1}{b-1}=\binom{n-a}{b-1}$ ways in which to arrange the $2$'s and hence there are $\binom{a-1}{b-1}\binom{n-a}{b-1}$ binary members of $\mathcal{W}_{n,a,b}$.

To enumerate the non-binary $\pi \in \mathcal{W}_{n,a,b}$, consider inserting exactly $a$ 1's (as $b$ runs) into $\rho \in \mathcal{W}_{n-a,c,d}$ for some $c$ and $d$, where $\rho$ is represented using the letters in $\{2,3,\ldots\}$. Observe that we may insert one or more $1$'s at the very end of $\rho$ or directly preceding any $2$ of $\rho$ except for the first $2$ in any non-initial runs of $2$'s.  To see this, note that inserting one or more 1's directly prior to the first $2$ in a non-initial run always introduces an occurrence of 31-2 in which the role of `3' is played by the predecessor of the 2 in $\rho$. Thus, there are $c-d+2$ potential sites altogether within $\rho$ in which one may insert a run of $1$'s.  Since $\pi$ must start with $1$ and contain exactly $b$ runs of $1$'s, there are $\binom{c-d+1}{b-1}$ choices for the sites of the runs within $\pi$.  Then the $a$ $1$'s are to be inserted into the $b$ chosen sites such that each receives at least one letter, which can be achieved in $\binom{a-1}{b-1}$ ways.  As there are $w_{n-a}(c,d)$ possibilities for $\rho$, allowing $c$ and $d$ to vary yields
$$\sum_{c=1}^{n-a-1}\sum_{d=1}^cw_{n-a}(c,d)\binom{a-1}{b-1}\binom{c-d+1}{b-1}$$
non-binary members of $\mathcal{W}_{n,a,b}$ altogether.  Combining this with the prior binary case yields \eqref{31-2lem1e1}.
\end{proof}

Define the generating function
$$w(t;x,y)=\sum_{n\geq2}\left(\sum_{a=1}^{n-1}\sum_{b=1}^aw_n(a,b)x^{a-1}y^{b-1}\right)t^n.$$
Then $w(t;x,y)$ satisfies the following functional equation.

\begin{lemma}\label{31-2lem2}
We have
\small\begin{align}
w(t;x,y)&=\frac{(1+x(y-1)t)t^2}{(1-xt)(1-(x+1)t-x(y-1)t^2)}+\frac{(1+x(y-1)t)t}{(1-xt)^2}w\left(t;\frac{1+x(y-1)t}{1-xt},\frac{1-xt}{1+x(y-1)t}\right).\label{31-2lem2e1}
\end{align}\normalsize
\end{lemma}
\begin{proof}
First note
\begin{align*}
&\sum_{n\geq2}t^n\sum_{a=1}^{n-1}x^{a-1}\sum_{b=1}^a\binom{a-1}{b-1}\binom{n-a}{b-1}y^{b-1}=\sum_{a\geq1}x^{a-1}t^a\sum_{b=1}^a\binom{a-1}{b-1}y^{b-1}\sum_{n\geq1}\binom{n}{b-1}t^n\\
&=\sum_{a\geq1}x^{a-1}t^a\sum_{b=1}^a\binom{a-1}{b-1}y^{b-1}\sum_{n\geq0}\binom{n}{b-1}t^n-\sum_{a\geq1}x^{a-1}t^a\\
&=\sum_{a\geq1}\frac{x^{a-1}t^a}{1-t}\sum_{b=1}^a\binom{a-1}{b-1}\left(\frac{yt}{1-t}\right)^{b-1}-\frac{t}{1-xt}=\sum_{a\geq1}\frac{x^{a-1}t^a(1+(y-1)t)^{a-1}}{(1-t)^a}-\frac{t}{1-xt}\\
&=\frac{t}{1-(x+1)t-x(y-1)t^2}-\frac{t}{1-xt}=\frac{(1+x(y-1)t)t^2}{(1-xt)(1-(x+1)t-x(y-1)t^2)}
\end{align*}
and
\begin{align*}
&\sum_{n\geq3}t^n\sum_{a=1}^{n-1}x^{a-1}\sum_{b=1}^ay^{b-1}\sum_{c=1}^{n-a-1}\sum_{d=1}^cw_{n-a}(c,d)\binom{a-1}{b-1}\binom{c-d+1}{b-1}\\
&=\sum_{a\geq1}x^{a-1}t^a\sum_{b=1}^ay^{b-1}\sum_{c\geq1}\sum_{d=1}^c \binom{a-1}{b-1}\binom{c-d+1}{b-1}\sum_{n\geq c+1}w_n(c,d)t^n\\
&=\sum_{b\geq1}y^{b-1}\sum_{c\geq1}\sum_{d=1}^c\binom{c-d+1}{b-1}\sum_{n\geq c+1}w_n(c,d)t^n\sum_{a\geq b}\binom{a-1}{b-1}x^{a-1}t^a\\
&=\sum_{b\geq1}\sum_{c\geq1}\sum_{d=1}^c\binom{c-d+1}{b-1}\sum_{n\geq c+1}w_n(c,d)t^{n+1}\cdot\frac{(xyt)^{b-1}}{(1-xt)^b}\\
&=\frac{t}{1-xt}\sum_{c\geq1}\sum_{d=1}^c\sum_{n\geq c+1}w_n(c,d)t^{n}\sum_{b=1}^{c-d+2}\binom{c-d+1}{b-1}\left(\frac{xyt}{1-xt}\right)^{b-1}\\
&=\frac{t}{1-xt}\sum_{c\geq1}\sum_{d=1}^c\sum_{n\geq c+1}w_n(c,d)t^{n}\cdot\left(\frac{1+x(y-1)t}{1-xt}\right)^{c-d+1}\\
&=\frac{(1+x(y-1)t)t}{(1-xt)^2}\sum_{n\geq2}\left(\sum_{c=1}^{n-1}\sum_{d=1}^cw_n(c,d)\left(\frac{1+x(y-1)t}{1-xt}\right)^{c-1}\left(\frac{1-xt}{1+x(y-1)t}\right)^{d-1}\right)t^n\\
&=\frac{(1+x(y-1)t)t}{(1-xt)^2}w\left(t;\frac{1+x(y-1)t}{1-xt},\frac{1-xt}{1+x(y-1)t}\right).
\end{align*}
Thus, multiplying both sides of \eqref{31-2lem1e1} by $x^{a-1}y^{b-1}t^n$, and summing over all $n \geq 3$ and $1 \leq b \leq a \leq n-1$, leads to \eqref{31-2lem2e1}.
\end{proof}

The generating function for $c_n(31\text{-}2)$ may be expressed in terms of Chebyshev polynomials as follows.

\begin{theorem}\label{31-2th}
We have
\begin{equation}\label{31-2the1}
\sum_{n\geq1}c_n(31\text{-}2)t^n=\frac{t}{1-t}+\sum_{i\geq1}\frac{\prod_{j=1}^i\left(1+\frac{U_{j-1}(s)}{U_j(s)}\right)}{U_{i+1}(s)},
\end{equation}
where $s=\frac{1-t}{2t}$.
\end{theorem}
\begin{proof}
By \eqref{31-2key}, we seek a formula for $\frac{t}{1-t}+w(t;1,1)$. Define the sequence of ordered pairs $(x_n,y_n)$ recursively by
\begin{equation}\label{31-2the1}
(x_{n+1},y_{n+1})=\left(\frac{1+tx_n(y_n-1)}{1-tx_n},\frac{1-tx_n}{1+tx_n(y_n-1)}\right), \qquad n \geq0,
\end{equation}
with $(x_0,y_0)=(1,1)$.  Then by \eqref{31-2lem2e1} and since $x_ny_n=1$, we have
$$w(t;x_n,y_n)=\frac{t^2x_{n+1}}{1-t-t^2-t(1-t)x_n}+\frac{tx_{n+1}}{1-tx_n}w(t;x_{n+1},y_{n+1}), \qquad n \geq0,$$
and iteration gives
\begin{equation}\label{31-2the2}
w(t;1,1)=\sum_{i\geq0}\frac{t^{i+2}\prod_{j=1}^{i+1}x_j}{(1-t-t^2-t(1-t)x_i)\prod_{j=0}^{i-1}(1-tx_j)}.
\end{equation}

Let $x_n=\frac{\alpha_n}{\beta_n}$ for $n \geq0$, where $\alpha_n=\alpha_n(t)$ and $\beta_n=\beta_n(t)$ are sequences of polynomials which we seek to determine.  First note
$$x_{n+1}=\frac{1+tx_n(y_n-1)}{1-tx_n}=\frac{1-tx_n}{1-tx_n}+\frac{tx_ny_n}{1-tx_n}=1+\frac{t}{1-tx_n}$$
so that
$$\frac{\alpha_{n+1}}{\beta_{n+1}}=x_{n+1}=1+\frac{t}{1-t\frac{\alpha_n}{\beta_n}}=\frac{(1+t)\beta_n-t\alpha_n}{\beta_n-t\alpha_n}.$$
We thus obtain the system of recurrences
\begin{equation}\label{31-2the3}
\alpha_{n+1}=(1+t)\beta_n-t\alpha_n, \quad \beta_{n+1}=\beta_n-t\alpha_n, \qquad n \geq0,
\end{equation}
with $\alpha_0=\beta_0=1$.  From \eqref{31-2the3}, we have $\beta_{n+1}=\frac{\alpha_{n+2}+t\alpha_{n+1}}{1+t}$, and hence
$$\frac{\alpha_{n+2}+t\alpha_{n+1}}{1+t}=\frac{\alpha_{n+1}+t\alpha_{n}}{1+t}-t\alpha_n,$$
which implies the recurrence
\begin{equation}\label{31-2the4}
\alpha_{n+2}=(1-t)\alpha_{n+1}-t^2\alpha_n, \qquad n \geq0,
\end{equation}
with $\alpha_0=\alpha_1=1$.  Similarly, we have that $\beta_n$ satisfies this same recurrence, but with initial values $\beta_0=1$ and $\beta_1=1-t$.

Thus, by \eqref{31-2the2}, we have
\begin{align}
w(t;1,1)&=\sum_{i\geq0}\frac{t^{i+2}\prod_{j=1}^{i+1}\alpha_j}{\beta_{i+1}((1-t-t^2)\beta_i-t(1-t)\alpha_i)\prod_{j=0}^{i-1}(\beta_j-t\alpha_j)}=\sum_{i\geq0}\frac{t^{i+2}\prod_{j=1}^{i+1}\alpha_j}{\beta_{i+1}\beta_{i+2}\prod_{j=0}^{i-1}\beta_{j+1}}\notag\\
&=\sum_{i\geq1}\frac{t^{i+1}\prod_{j=1}^i\alpha_j}{\prod_{j=1}^{i+1}\beta_j}, \label{31-2the5}
\end{align}
where we have used the facts $\beta_{j+1}=\beta_j-t\alpha_j$ and $\beta_{i+2}=(1-t-t^2)\beta_i-t(1-t)\alpha_i$ for all $i,j\geq0$, which can be shown by induction.  A comparison with the initial values and recurrence of the quantity $\gamma_n=\frac{\alpha_n}{t^n}$, which can be obtained from \eqref{31-2the4}, reveals
$U_n(s)+U_{n-1}(s)=\gamma_n$, and hence $\alpha_n=t^n(U_n(s)+U_{n-1}(s))$ for all $n \geq1$, where $s=\frac{1-t}{2t}$. Similarly, we obtain $\beta_n=t^nU_n(s)$ for all $n$.  Substituting this into \eqref{31-2the5}, we get
$$w(t;1,1)=\sum_{i\geq1}\frac{t^{i+1}\prod_{j=1}^it^j(U_j(s)+U_{j-1}(s))}{\prod_{j=1}^{i+1}t^jU_j(s)}=\sum_{i\geq1}\frac{\prod_{j=1}^i\left(1+\frac{U_{j-1}(s)}{U_j(s)}\right)}{U_{i+1}(s)}, $$
which yields the desired formula.
\end{proof}

\noindent \emph{Remark:} It is possible to find a formula similar to \eqref{31-2the5} for $w(t;x,y)$ for all $x,y$ in some interval containing $1$ for all $t$ sufficiently small, say, $|t|\leq \frac{1}{16}$.  Note that the corresponding polynomials $\alpha_n$ and $\beta_n$ will satisfy the same recurrences as before for $n \geq 3$, but with the initial values $\alpha_1=1+x(y-1)t,\,\alpha_2=1-xt-xyt^2$ and $\beta_1=1-xt,\, \beta_2=1-(x+1)t-x(y-1)t^2$.

\subsection{The pattern 11-2}

To deal with this case, we first consider two preliminary arrays.  Given $n \geq1$ and $1 \leq a \leq n$, let $m_n(a)$ denote the cardinality of the subset of $\mathcal{C}_n$ whose members contain no levels and end in $a$.  Put $m_n(a)=0$ if $n \leq 0$ or $a \notin [n]$.  Let $p_n(a)$ for $a,n \geq 1$ denote the number of words $w=w_1\cdots w_n$ on the alphabet $[a]$ that satisfy $w_{i+1} \leq w_i+1$ for $1 \leq i \leq n-1$ and avoid 11-2, with $p_0(a):=1$ for all $a \geq 1$.  Given $n \geq 1$ and $1 \leq b \leq a$, let $p_n(a,b)$ enumerate the subset of those words counted by $p_n(a)$ which start with $b$.  The arrays $m_n(a)$ and $p_n(a,b)$ are given recursively as follows.

\begin{lemma}\label{11-2lem1}
If $n \geq 2$, then
\begin{equation}\label{11-2leme1}
m_n(a)=m_{n-1}(a-1)+\sum_{i=a+1}^{n-1}m_{n-1}(i), \qquad 1 \leq a \leq n-1,
\end{equation}
with $m_n(n)=1$ for all $n \geq 1$. If $n\geq2$, then
\begin{equation}\label{11-2leme2}
p_n(a,b)=p_{n-1}(a,b+1)+p_{n-2}(b)+\sum_{i=1}^{b-1}p_{n-1}(a,i), \qquad 1 \leq b \leq a-1,
\end{equation}
with $p_n(a,a)=p_{n-1}(a)$ for all $a,n \geq1$, where $p_n(a)=\sum_{b=1}^ap_n(a,b)$ for $n\geq 1$ and $p_0(a)=1$.
\end{lemma}
\begin{proof}
Considering the penultimate letter within a member of $\mathcal{C}_n$ enumerated by $m_n(a)$ implies \eqref{11-2leme1}.  Let $\mathcal{P}_{n,a,b}$ denote the class of words enumerated by $p_n(a,b)$ and $\mathcal{P}_{n,a}=\cup_{b=1}^a\mathcal{P}_{n,a,b}$.  To show \eqref{11-2leme2}, consider the second letter $i$ within $w=w_1\cdots w_n \in \mathcal{P}_{n,a,b}$, where $n \geq 2$ and $b<a$.  If $i=b+1$ or $i \in [b-1]$, then there are clearly $p_{n-1}(a,b+1)$ and $\sum_{i=1}^{b-1}p_{n-1}(a,i)$ possibilities, respectively.  On the other hand, if $i=b$, then no letters in $[b+1,a]$ can occur in $w$ and deleting the initial two $b$'s from $w$ results in an arbitrary member of $\mathcal{P}_{n-2,b}$.  Hence, there are $p_{n-2}(b)$ possibilities when $i=b$, and combining with the previous cases yields \eqref{11-2leme2}.  Finally, the values of $m_n(n)$ and $p_n(a,a)$ for $n \geq1$ follow from the definitions.
\end{proof}

The sequence $c_n(11\text{-}2)$ can be given in terms of the prior two arrays as follows.

\begin{lemma}\label{11-2lem2}
If $n \geq1$, then
\begin{equation}\label{11-2lem2e1}
c_n(11\text{-}2)=M_{n-1}+\sum_{i=1}^{n-1}\sum_{j=1}^im_i(j)p_{n-i-1}(j),
\end{equation}
where $m_k(j)$ and $p_k(j)$ are as in Lemma \ref{11-2lem1}.
\end{lemma}
\begin{proof}
Let $\pi=\pi_1\cdots \pi_n \in \mathcal{C}_n(11\text{-}2)$.  If $\pi$ contains no levels, then there are $M_{n-1}$ possibilities for $\pi$, as previously noted.  So assume $\pi$ contains at least one level, with $i$ the smallest index corresponding to one and $\pi_i=\pi_{i+1}=j$.  Then $\pi$ can be decomposed as $\pi=\pi'j\pi''$, where $\pi'$ of length $i$ contains no levels and ends in $j$ and $\pi''$ is possibly empty.  Note $\pi''$ avoids 11-2 and is $j$-ary, due to the level involving the letter $j$ directly preceding it.  Thus, for each $i$ and $j$, there are $m_i(j)$ and $p_{n-i-1}(j)$ possibilities for $\pi'$ and $\pi''$, respectively, and allowing $i$ and $j$ to vary  yields all members of $\mathcal{C}_n(11\text{-}2)$ containing at least one level, which implies \eqref{11-2lem2e1}.
\end{proof}

Let $M_n(v)=\sum_{a=1}^nm_n(a)v^{a-1}$ for $n \geq1$ and $M(t;v)=\sum_{n\geq1}M_n(v)t^n$. Define
$$P(t;v,u)=\sum_{n\geq1}\sum_{a\geq1}\sum_{b=1}^ap_n(a,b)v^{a-1}u^{b-1}t^n.$$
Then $M(t;v)$ and $P(t;v,u)$ satisfy the following functional equations.

\begin{lemma}\label{11-2lem3}
We have
\begin{equation}\label{11-2lem3e1}
\left(1-tv+\frac{t}{1-v}\right)M(t;v)=t+\frac{t}{1-v}M(t;1)
\end{equation}
and
\begin{equation}\label{11-2lem3e2}
\left(1-\frac{t}{u}-\frac{tu}{1-u}\right)P(t;v,u)=\frac{t(1+t)}{(1-v)(1-uv)}+\left(\frac{t^2}{1-v}-\frac{tu}{1-u}\right)P(t;uv,1)-\frac{t}{u}P(t;v,0).
\end{equation}
\end{lemma}
\begin{proof}
Multiplying both sides of \eqref{11-2leme1} by $t^nv^{a-1}$, and summing over all $n \geq 2$ and $1 \leq a \leq n-1$, yields
$$M(t;v)=t+tvM(t;v)
+\frac{t}{1-v}(M(t;1)-M(t;v)),$$
which implies \eqref{11-2lem3e1}. Similarly, from \eqref{11-2leme2}, we get
\begin{align*}
P(t;v,u)&=t\sum_{a\geq1}\sum_{b=1}^a p_1(a,b)v^{a-1}u^{b-1}+\sum_{n\geq2}\sum_{a\geq1}\sum_{b=1}^a p_n(a,b)v^{a-1}u^{b-1}t^n\\
&=t(1+t)\sum_{a\geq1}v^{a-1}\sum_{b=1}^au^{b-1}
+\sum_{n\geq2}\sum_{a\geq2}\sum_{b=1}^{a-1}p_{n-1}(a,b+1)v^{a-1}u^{b-1}t^n\\
&\quad+\sum_{n\geq3}\sum_{a\geq1}\sum_{b=1}^a\sum_{i=1}^bp_{n-2}(b,i)v^{a-1}u^{b-1}t^n
+\sum_{n\geq2}\sum_{a\geq2}\sum_{b=2}^a\sum_{i=1}^{b-1}p_{n-1}(a,i)v^{a-1}u^{b-1}t^n\\
&=\frac{t(1+t)}{(1-v)(1-uv)}+\frac{t}{u}(P(t;v,u)-P(t;v,0))+\frac{t^2}{1-v}P(t;uv,1)\\
&\quad+\frac{tu}{1-u}(P(t;v,u)-P(t;uv,1)),
\end{align*}
which rearranges to give \eqref{11-2lem3e2}.
\end{proof}

The functional equations from the prior lemma may be solved explicitly.

\begin{lemma}\label{11-2lem4}
We have
\begin{equation}\label{11-2lem4e1}
M(t;v)=\frac{1+(1-2v)t-\sqrt{1-2t-3t^2}}{2(1-v+(1-v)t+v^2t)}
\end{equation}
and
\begin{equation}\label{11-2lem4e2}
P(t;v,1)=\sum_{j\geq0}b\left(\frac{t^jT^{2j}(t)v}{(1+t)^j}\right)\prod_{i=0}^{j-1}a\left(\frac{t^iT^{2i}(t)v}{(1+t)^i}\right),
\end{equation}
where
$$a(v)=\frac{(t-v)t^3T^8(t)}{(1+t)^3(1+t-vT^2(t))}, \quad b(v)=\frac{tT^3(t)(2-T(t))}{(1-v)(1+v-vT(t))(1+t-vT^2(t))}$$
and
$$T(t)=\frac{1+t-\sqrt{1-2t-3t^2}}{2t}.$$
\end{lemma}
\begin{proof}
Letting $v=v_0=\frac{1+t-\sqrt{1-2t-3t^2}}{2t}$ in \eqref{11-2lem3e1} gives $M(t;1)=v_0-1$, and substituting this back into \eqref{11-2lem3e1} implies \eqref{11-2lem4e1}.  Letting $u=\frac{1}{T}$ in \eqref{11-2lem3e2} and solving for $P(t;v,0)$, where $T=T(t)$ is as defined, yields
\begin{equation}\label{11-2lem4e3}
P(t;v,0)=\frac{1}{T}\left(\frac{t}{1-v}-\frac{1}{T-1}\right)P(t;v/T,1)+\frac{1+t}{(1-v)(T-v)}.
\end{equation}
Substituting \eqref{11-2lem4e3} into \eqref{11-2lem3e2}, we find
\begin{align}
\left(1-\frac{t}{u}-\frac{tu}{1-u}\right)P(t;v,u)&=\frac{t(t+1)}{(1-v)(1-uv)}+\left(\frac{t^2}{1-v}-\frac{tu}{1-u}\right)P(t;uv,1)\notag\\
&\quad-\frac{t}{u}\left(\frac{1}{T}\left(\frac{t}{1-v}-\frac{1}{T-1}\right)P(t;v/T,1)+\frac{1+t}{(1-v)(T-v)}\right).\label{11-2lem4e4}
\end{align}

To solve \eqref{11-2lem4e4}, we let $u$ be the other root of the kernel equation $1-\frac{t}{u}-\frac{tu}{1-u}=0$, namely, $u=\frac{tT}{1+t}$, to obtain
\begin{align*}
&\frac{1+t}{T}\left(\frac{1}{T}\left(\frac{t}{1-v}-\frac{1}{T-1}\right)P(t;v/T,1)
+\frac{1+t}{(1-v)(T-v)}\right)\\
&=\left(\frac{t^2}{1-v}-\frac{t^2T}{1+t-tT}\right)P\left(t;\frac{vtT}{1+t},1\right)+\frac{t(1+t)^2}{(1-v)(1+t-vtT)}.
\end{align*}
Replacing $v$ by $vT$, we have
\begin{align*}
\frac{1+t}{T^2}\left(\frac{t}{1-vT}-\frac{1}{T-1}\right)P(t;v,1)&=\left(\frac{t^2}{1-vT}-\frac{t^2T}{1+t-tT}\right)P\left(t;\frac{vtT^2}{1+t},1\right)\\
&\quad+\frac{(1+t)^2}{1-vT}\left(\frac{t}{1+t-vtT^2}-\frac{1}{(1-v)T^2}\right).
\end{align*}
After several algebraic steps and frequent use of the fact $tT^2=(t+1)(T-1)$, the last equation can be written as
\begin{equation}\label{11-2lem4e5}
P(t;v,1)=a(v)P\left(t;\frac{vtT^2}{1+t},1\right)+b(v),
\end{equation}
where $a(v)$ and $b(v)$ are as defined.  Iteration of \eqref{11-2lem4e5} for $t$ sufficiently close to zero then yields \eqref{11-2lem4e2}.
\end{proof}

One can now obtain an explicit formula for the generating function of $c_n(11\text{-}2)$.

\begin{theorem}\label{11-2th}
We have
\begin{equation}\label{11-2the1}
\sum_{n\geq1}c_n(11\text{-}2)t^n
=\frac{(1+t)(1-t-\sqrt{1-2t-3t^2})}{2t}
+\frac{t^2T(t)}{1+t}\sum_{j\geq1}b\left(\frac{t^{j}T^{2j-1}(t)}{(1+t)^{j}}\right)\prod_{i=1}^{j-1}a\left(\frac{t^{i}T^{2i-1}(t)}{(1+t)^{i}}\right),
\end{equation}
where $a(v)$, $b(v)$ and $T(t)$ are as in Lemma \ref{11-2lem4}.
\end{theorem}
\begin{proof}
Note first that \eqref{11-2lem4e1} may be written as
$$M(t;v)=\frac{2t}{(1+t+\sqrt{1-2t-3t^2})\left(1-\frac{2tv}{1+t+\sqrt{1-2t-3t^2}}\right)},$$
and hence
$$\sum_{n\geq j}m_n(j)t^n=[v^{j-1}]M(t;v)=\frac{(2t)^j}{(1+t+\sqrt{1-2t-3t^2})^j}, \qquad j \geq 1.$$
Thus, by \eqref{11-2lem2e1}, we have
\begin{align*}
\sum_{n\geq1}c_n(11\text{-}2)t^n-\sum_{n\geq1}M_{n-1}t^n&=\sum_{n\geq2}\sum_{i=1}^{n-1}\sum_{j=1}^im_i(j)p_{n-i-1}(j)t^n\\
&=t\sum_{j\geq1}\sum_{n\geq0}\left(\sum_{i\geq j}m_i(j)t^i\right)p_{n}(j)t^n\\
&=t\sum_{j\geq1}\sum_{n\geq0}\frac{(2t)^j}{(1+t+\sqrt{1-2t-3t^2})^j}p_{n}(j)t^n\\
&=\frac{2t^2}{1+t+\sqrt{1-2t-3t^2}}\sum_{n\geq0}\sum_{j\geq1}p_{n}(j)\frac{(2t)^{j-1}t^n}{(1+t+\sqrt{1-2t-3t^2})^{j-1}}\\
&=\frac{2t^2}{1+t+\sqrt{1-2t-3t^2}}\biggl(
\sum_{j\geq1}\left(\frac{2t}{1+t+\sqrt{1-2t-3t^2}}\right)^{j-1}\\
&\quad+P\left(t;\frac{2t}{1+t+\sqrt{1-2t-3t^2}},1\right)\biggr).
\end{align*}
Hence, we get
\begin{align*}
\sum_{n\geq1}c_n(11\text{-}2)t^n&=\frac{(1+t)(1-t-\sqrt{1-2t-3t^2})}{2t}\\
&\quad+\frac{2t^2}{1+t+\sqrt{1-2t-3t^2}}P\left(t;\frac{2t}{1+t+\sqrt{1-2t-3t^2}},1\right)\\
&=\frac{(1+t)(1-t-\sqrt{1-2t-3t^2})}{2t}+\frac{t^2T}{1+t}P\left(t;\frac{tT}{1+t},1\right),
\end{align*}
where we have made use of the fact $\sum_{n\geq0}M_nt^n=\frac{1-t-\sqrt{1-2t-3t^2}}{2t^2}$.  Thus, by \eqref{11-2lem4e2}, we obtain
\begin{align*}
\sum_{n\geq1}c_n(11\text{-}2)t^n
&=\frac{(1+t)(1-t-\sqrt{1-2t-3t^2})}{2t}
+\frac{t^2T}{1+t}\sum_{j\geq0}b\left(\frac{t^{j+1}T^{2j+1}}{(1+t)^{j+1}}\right)\prod_{i=0}^{j-1}a\left(\frac{t^{i+1}T^{2i+1}}{(1+t)^{i+1}}\right),
\end{align*}
which implies \eqref{11-2the1}.
\end{proof}

\subsection{The pattern 21-1}

To determine a formula for the generating function of the sequence $c_n(21\text{-}1)$ for $n \geq 1$, we refine it as follows by considering a certain pair of parameters on $\mathcal{C}_n(21\text{-}1)$.  Recall that if $\pi=\pi_1\cdots \pi_n$ is a sequence with $\pi_i=x$ and $\pi_{i+1}=y$ for some $i\in [n-1]$ where $x>y$, then $\pi$ has a \emph{descent} at index $i$ and the values $x$ and $y$ are referred to as the \emph{top} and \emph{bottom} of the descent, respectively.  Given $n \geq 3$ and $1 \leq a <m \leq n-1$, let $r_n(m,a)$ denote the number of members of $\mathcal{C}_n(21\text{-}1)$ that contain at least one descent and have greatest letter $m$ and smallest descent bottom $a$, with $r_n(m,a)$ taken to be zero otherwise.  For example, we have $r_6(4,2)=7$, the enumerated sequences belonging to $\{112342, 122342,123234,123342, 123423, 123432, 123442\}$.

Given $n \geq1$ and $1 \leq m \leq n$, let $r_n(m)$ be the number of members of $\mathcal{C}_n(21\text{-}1)$ having largest letter $m$.  By the definitions, we have $c_n(21\text{-}1)=\sum_{m=1}^nr_m(n)$ and
\begin{equation}\label{rn(m)form}
r_n(m)=\binom{n-1}{m-1}+\sum_{a=1}^{m-1}r_n(m,a), \qquad 1 \leq m \leq n-1,
\end{equation}
with $r_n(n)=1$ for all $n \geq1$.  Note that the  $\binom{n-1}{m-1}$ term in \eqref{rn(m)form} accounts for the weakly increasing members of $\mathcal{C}_n(21\text{-}1)$ whose largest letter is $m$.

The array $r_n(m,a)$ is given recursively as follows.

\begin{lemma}\label{21-1lem1}
If $n\geq4$, then
\begin{equation}\label{21-1lem1e1}
r_n(m,a)=\sum_{i=1}^{n-m}r_{n-i}(m-1,a-1), \qquad 2 \leq a <m \leq n-1,
\end{equation}
and
\begin{align}
r_n(m,1)&=r_{n-1}(m,1)+r_{n-2}(m-1)+\sum_{i=1}^{m-2}\sum_{j=i}^{n-m-1}\binom{j-1}{i-1}r_{n-j-2}(m-1)\notag\\
&\quad+\sum_{i=1}^{m-1}\sum_{j=m-1}^{n-i-2}\binom{j-1}{m-2}r_{n-j-2}(i)+\sum_{i=2}^{m-2}\sum_{j=m}^{n-3}\sum_{k=1}^\ell r_j(m-1,i)r_{n-j-2}(k), \quad 2 \leq m \leq n-1, \label{21-1lem1e2}
\end{align}
with initial values $r_1(1)=r_2(1)=r_2(2)=r_3(2,1)=1$, where $\ell=\min\{i-1,n-j-2\}$ and $r_n(m)$ is given by \eqref{rn(m)form}.
\end{lemma}
\begin{proof}
The initial values for $1 \leq n \leq 3$ are easily seen to hold, so we may assume $n \geq4$.  Let $\mathcal{R}_{n,m,a}$ and $\mathcal{R}_{n,m}$ denote the subsets of $\mathcal{C}_n(21\text{-}1)$ enumerated by $r_n(m,a)$ and $r_n(m)$, respectively.  To realize \eqref{21-1lem1e1}, first note that a member of $\mathcal{R}_{n,m,a}$ where $a>1$ cannot contain any $1$'s outside of the initial run, for otherwise its smallest descent bottom would be $1$ and not $a$.  Thus, members of $\mathcal{R}_{n,m,a}$ where $a>1$ must be of the form $\pi=1^i\pi'$, where $i \geq 1$ and $\pi'$ is such that $\pi'-1$ (the sequence obtained by subtracting 1 from each entry of $\pi'$) belongs to $\mathcal{R}_{n-i,m-1,a-1}$.  Considering all possible $i$ then gives \eqref{21-1lem1e1}.

Now suppose $\pi \in \mathcal{R}_{n,m,1}$, where $n \geq 4$ and $2 \leq m \leq n-1$.  If $\pi$ starts with two or more 1's, then there are clearly $r_{n-1}(m,1)$ possibilities, upon deleting the second $1$, so assume $\pi$ starts with a single 1.  Then $\pi$ avoiding 21-1 and having 1 as a descent bottom implies $\pi$ must contain exactly two (non-consecutive) 1's in this case.  If $\pi$ is of the form $\pi=1\alpha1$, where $\alpha$ contains no 1's, then there are $r_{n-2}(m-1)$ such $\pi$.  So assume $\pi=1\alpha1\beta$, where $\alpha$ and $\beta$ are nonempty. First suppose $\alpha$ is weakly increasing.  If $\alpha$ has greatest letter $i+1$ for some $1 \leq i \leq m-2$ with $|\alpha|=j$, then there are $\binom{j-1}{i-1}$ possibilities for $\alpha$.  Then $\beta$ is such that $\beta-1 \in \mathcal{R}_{n-j-2,m-1}$ , where $i \leq j \leq n-m-1$, and hence there are $r_{n-j-2}(m-1)$ possible $\beta$.  Summing over all $i$ and $j$ then gives $\sum_{i=1}^{m-2}\sum_{j=i}^{n-m-1}\binom{j-1}{i-1}r_{n-j-2}(m-1)$ such $\pi$ wherein $\alpha$ does not contain $m$.  On the other hand, if $\alpha$ contains $m$, then there are $\binom{j-1}{m-2}$ possibilities for $\alpha$ where $|\alpha|=j$.  This implies $\beta-1 \in \mathcal{R}_{n-j-2,i}$ for some $i \in [m-1]$.  Summing over all $i$ and $j$ then gives $\sum_{i=1}^{m-1}\sum_{j=m-1}^{n-i-2}\binom{j-1}{m-2}r_{n-j-2}(i)$ possible $\pi=1\alpha1\beta$ wherein $\alpha$ is weakly increasing and contains $m$.

Finally, assume $\alpha$ within $\pi=1\alpha1\beta$ contains at least one descent.  Suppose that $i+1$ is the smallest descent bottom of $\pi$ lying within $\alpha$.  Note that $i \geq 2$, and hence $m\geq 4$ in this case with $i \in [2,m-2]$, for otherwise, if $i=1$ (i.e., $\alpha$ contains a descent bottom of size 2), then $\beta$ would have to be empty, contrary to assumption.  Then $\beta$ must have largest letter $k+1$ for some $1 \leq k \leq i-1$, and hence $\beta-1 \in \mathcal{R}_{n-j-2,k}$, where $j=|\alpha|$, for if not, then $\pi$ would contain an occurrence of 21-1 in which the role of `1' is played by $i+1$ and the role of `2' is played by the descent top of $i+1$ within $\alpha$. In particular, it is seen that if $\alpha$ contains a descent, then $\beta$ cannot contain $m$.  Thus, $\alpha$ must contain $m$, i.e., $\alpha-1\in\mathcal{R}_{j,m-1,i}$.  Conversely, $\pi=1\alpha1\beta$ with the stated restrictions on $\alpha$ and $\beta$ is seen to avoid 21-1.  Note that $j\geq m$ in order for $\alpha$ to contain both $m$ and a descent and $j \leq n-3$ in order for $\beta$ to be nonempty.  Summing over all $i$, $j$ and $k$ then gives $\sum_{i=2}^{m-2}\sum_{j=m}^{n-3}\sum_{k=1}^\ell r_j(m-1,i)r_{n-j-2}(k)$ possible $\pi$ in this case.  Combining with the prior cases then yields \eqref{21-1lem1e2} and completes the proof.
\end{proof}

Define the generating functions $R(t;v)=\sum_{n\geq1}\sum_{m=1}^nr_n(m)t^nv^{m-1}$ and $$R(t;v,u)=\sum_{n\geq3}\sum_{m=2}^{n-1}\sum_{a=1}^{m-1}
r_n(m,a)t^nv^{m-1}u^{a-1}.$$ In order to find an explicit formula for $R(t;v)$, we will need the following relation.

\begin{lemma}\label{idr2r3}
For all $k\geq1$,
\begin{equation}\label{idr2r3e1}
\sum_{n\geq3}\sum_{m=2}^{n-1}\sum_{a=1}^{m-1}
r_{n+k}(m+k,a+k)t^nv^{m-1}=\frac{R(t;v,1)}{(1-t)^k}.
\end{equation}
\end{lemma}
\begin{proof}
Equating coefficients of $t^nv^{m-1}$ in
$$\sum_{n\geq3}\sum_{m=2}^{n-1}\sum_{a=1}^{m-1}r_n(m,a)t^nv^{m-1}=(1-t)^k\sum_{n\geq3}\sum_{m=2}^{n-1}\sum_{a=1}^{m-1}
r_{n+k}(m+k,a+k)t^nv^{m-1},$$
we must show, for each $n \geq 3$ and $2 \leq m \leq n-1$, the identity
\begin{equation}\label{idr2r3e2}
\sum_{a=1}^{m-1}r_n(m,a)=\sum_{a=1}^{m-1}\sum_{\ell=0}^k(-1)^{k-\ell}\binom{k}{\ell}r_{n+\ell}(m+k,a+k), \qquad k \geq 1.
\end{equation}
To establish \eqref{idr2r3e2}, it suffices to show
\begin{equation}\label{idr2r3e3}
r_n(m,a)=\sum_{\ell=0}^k(-1)^{k-\ell}\binom{k}{\ell}r_{n+\ell}(m+k,a+k), \qquad k \geq1.
\end{equation}
To do so, first note that the $k=1$ case of \eqref{idr2r3e3} is given by
\begin{equation}\label{idr2r3e4}
r_n(m,a)=r_{n+1}(m+1,a+1)-r_n(m+1,a+1), \quad 1 \leq a \leq m-1.
\end{equation}

For \eqref{idr2r3e4}, a combinatorial explanation can be given by considering whether a member of the set $\mathcal{R}_{n+1,m+1,a+1}$, where $n \geq m+1$, starts with a single 1 or two or more $1$'s.  Note that there are $r_n(m,a)$ possibilities in the former case and $r_n(m+1,a+1)$  in the latter, as a member of $\mathcal{R}_{n+1,m+1,a+1}$ cannot contain any $1$'s outside of its initial run of $1$'s. Further, we have that \eqref{idr2r3e4} holds trivially for $1 \leq n \leq m$, and hence \eqref{idr2r3e4} holds for all $n \geq 1$ for each fixed $m$ and $a$.  This establishes the $k=1$ case of \eqref{idr2r3e3} for all $n\geq1$.  An induction on $k$ may now be given for \eqref{idr2r3e3} by replacing each $r_{n+\ell}(m+k,a+k)$ term in the sum on the right side with
$$r_{n+\ell+1}(m+k+1,a+k+1)-r_{n+\ell}(m+k+1,a+k+1)$$
and gathering like terms.  Making use of the binomial recurrence $\binom{k+1}{\ell}=\binom{k}{\ell}+\binom{k}{\ell-1}$, the result of this substitution is
$$\sum_{\ell=0}^{k+1}(-1)^{k+1-\ell}\binom{k+1}{\ell}r_{n+\ell}(m+k+1,a+k+1).$$
This establishes the $k+1$ case of \eqref{idr2r3e3}, which completes the induction and proof.
\end{proof}

 Rewriting relation \eqref{rn(m)form} in terms of generating functions, we have
\begin{align*}
R(t;v)&=\sum_{n\geq1}r_n(n)t^nv^{n-1}+\sum_{n\geq2}\sum_{m=1}^{n-1}r_n(m)t^nv^{m-1}\notag\\
&=\frac{t}{1-vt}+\sum_{n\geq2}\sum_{m=1}^{n-1}\binom{n-1}{m-1}t^nv^{m-1}+\sum_{n\geq3}\sum_{m=2}^{n-1}\sum_{a=1}^{m-1}r_n(m,a)t^nv^{m-1}\notag\\
&=\frac{t}{1-vt}+\frac{t^2}{(1-vt)(1-(1+v)t)}+R(t;v,1),
\end{align*}
which implies
\begin{align}
R(t;v)&=\frac{t}{1-(1+v)t}+R(t;v,1).\label{eq211a1}
\end{align}

We now obtain the following functional equation for $R(t;v,u)$.

\begin{lemma}\label{21-1lem3}
We have
\begin{align}\label{eq211a3}
\left(1-\frac{vut}{1-t}\right)R(t;v,u)
&=\frac{vt^3(1-t)}{(1-(1+v)t)((1-t)^2-vt^2)}+\frac{v^2t^3}{1-t}R(t;v,1)R\left(t;\frac{vt}{1-t},1\right)\notag\\
&\quad-\frac{vt^2(vt^3-(1-t)^3)}{(1-2t)((1-t)^2-vt^2)}R(t;v,1)\notag\\
&\quad-\frac{v(1-v)t^4}{(1-2t)(1-(1+v)t)}R\left(t;\frac{vt}{1-t},1\right)+tR(t;v,0).
\end{align}
\end{lemma}
\begin{proof}
Define $\ell_{n,i,j}=\min\{i-1,n-j-2\}$.  By \eqref{21-1lem1e1} and \eqref{21-1lem1e2}, we have
\begin{align*}
&R(t;v,u)\\
&=\sum_{n\geq4}\sum_{m=3}^{n-1}\sum_{a=2}^{m-1}r_n(m,a)t^nv^{m-1}u^{a-1}
+\sum_{n\geq3}\sum_{m=2}^{n-1}r_n(m,1)t^nv^{m-1}\\
&=vu\sum_{m\geq2}\sum_{i\geq1}\sum_{n\geq m+i+1}\sum_{a=1}^{m-1}r_{n-i}(m,a)t^nv^{m-1}u^{a-1}
+\sum_{n\geq3}\sum_{m=2}^{n-1}r_n(m,1)t^nv^{m-1}\\
&=vu\sum_{m\geq2}\sum_{i\geq1}\sum_{n\geq m+1}\sum_{a=1}^{m-1}r_{n}(m,a)t^{n+i}v^{m-1}u^{a-1}
+\sum_{n\geq4}\sum_{m=2}^{n-2}r_{n-1}(m,1)t^nv^{m-1}\\
&\quad+\sum_{n\geq3}\sum_{m=2}^{n-1}r_{n-2}(m-1)t^nv^{m-1}
+\sum_{n\geq5}\sum_{m=3}^{n-2}\sum_{i=1}^{m-2}\sum_{j=i}^{n-m-1}\binom{j-1}{i-1}r_{n-j-2}(m-1)t^nv^{m-1}\\
&\quad+\sum_{n\geq4}\sum_{m=2}^{n-2}\sum_{i=1}^{m-1}\sum_{j=m-1}^{n-i-2}\binom{j-1}{m-2}r_{n-j-2}(i)t^nv^{m-1}\\
&\quad+\sum_{n\geq7}\sum_{m=4}^{n-3}\sum_{i=2}^{m-2}\sum_{j=m}^{n-3}\sum_{k=1}^{\ell_{n,i,j}}r_j(m-1,i)r_{n-j-2}(k)t^nv^{m-1}\\
&=\frac{vut}{1-t}\sum_{m\geq2}\sum_{n\geq m+1}\sum_{a=1}^{m-1}r_{n}(m,a)t^nv^{m-1}u^{a-1}
+tR(t;v,0)+vt^2R(t;v)\\
&\quad+v\sum_{i\geq1}\sum_{m\geq i+1}\sum_{j\geq i}\sum_{n\geq m}\binom{j-1}{i-1}r_n(m)t^{n+j+2}v^{m-1}
\\
&\quad+\sum_{i\geq1}\sum_{m\geq i+1}\sum_{j\geq m-1}\sum_{n\geq i+j+2}\binom{j-1}{m-2}r_{n-j-2}(i)t^nv^{m-1}\\
&\quad+\sum_{n\geq1}\sum_{i\geq2}\sum_{m\geq i+2}\sum_{j\geq m}\sum_{k=1}^{\min\{i-1,n\}}r_j(m-1,i)r_n(k)t^{n+j+2}v^{m-1}\\
&=\frac{vut}{1-t}\sum_{n\geq3}\sum_{m=2}^{n-1}\sum_{a=1}^{m-1}r_{n}(m,a)t^nv^{m-1}u^{a-1}
+tR(t;v,0)+vt^2R(t;v)\\
&\quad+v\sum_{i\geq1}\sum_{m\geq i+1}\sum_{n\geq m}r_n(m)\frac{t^{n+i+2}v^{m-1}}{(1-t)^i}
+\sum_{i\geq1}\sum_{m\geq i+1}\sum_{j\geq m-1}\sum_{n\geq i}\binom{j-1}{m-2}r_n(i)t^{n+j+2}v^{m-1}\\
&\quad+v\sum_{n\geq1}\sum_{k=1}^n\sum_{i\geq k+1}\sum_{m\geq i+1}\sum_{j\geq m+1}r_j(m,i)r_n(k)t^{n+j+2}v^{m-1}\\
&=\frac{vut}{1-t}R(t;v,u)
+tR(t;v,0)+vt^2R(t;v)+\frac{v}{1-2t}\sum_{n\geq 1}\sum_{m=1}^nr_n(m)\left(t^{n+3}-\frac{t^{n+m+2}}{(1-t)^{m-1}}\right)v^{m-1}\\
&\quad+\sum_{n\geq1}\sum_{i=1}^n\sum_{m\geq i+1}\sum_{j\geq m-1}\binom{j-1}{m-2}r_n(i)t^{n+j+2}v^{m-1}\\
&\quad+v\sum_{n\geq1}\sum_{k=1}^n\sum_{j\geq3}\sum_{m=2}^{j-1}\sum_{i=1}^{m-1}r_{j+k}(m+k,i+k)r_n(k)t^{n+j+k+2}v^{m+k-1}\\
&=\frac{vut}{1-t}R(t;v,u)
+tR(t;v,0)+vt^2R(t;v)+\frac{vt^3}{1-2t}\left(R(t;v)-R\left(t;\frac{vt}{1-t}\right)\right)\\
&\quad+\sum_{n\geq1}\sum_{i=1}^n\sum_{m\geq i+1}r_n(i)\frac{t^{n+m+1}v^{m-1}}{(1-t)^{m-1}}\\
&\quad+vt^2\sum_{n\geq1}\sum_{k=1}^nr_n(k)t^{n+k}v^k\left(\sum_{j\geq3}\sum_{m=2}^{j-1}\sum_{i=1}^{m-1}r_{j+k}(m+k,i+k)t^jv^{m-1}\right)\\
&=\frac{vut}{1-t}R(t;v,u)
+tR(t;v,0)+vt^2R(t;v)+\frac{vt^3}{1-2t}\left(R(t;v)-R\left(t;\frac{vt}{1-t}\right)\right)\\
&\quad+\frac{1}{1-(1+v)t}\sum_{n\geq1}\sum_{i=1}^n
r_n(i)\frac{t^{n+i+2}v^i}{(1-t)^{i-1}}\\
&\quad+vt^2\sum_{n\geq1}\sum_{k=1}^nr_n(k)t^{n+k}v^k\left(\sum_{j\geq3}\sum_{m=2}^{j-1}\sum_{i=1}^{m-1}r_{j+k}(m+k,i+k)t^jv^{m-1}\right)\\
&=\frac{vut}{1-t}R(t;v,u)
+tR(t;v,0)+vt^2R(t;v)+\frac{vt^3}{1-2t}\left(R(t;v)-R\left(t;\frac{vt}{1-t}\right)\right)\\
&\quad+\frac{vt^3}{1-(1+v)t}R\left(t;\frac{vt}{1-t}\right)+vt^2\sum_{n\geq1}\sum_{k=1}^nr_n(k)t^{n+k}v^k\left(\sum_{j\geq3}\sum_{m=2}^{j-1}\sum_{i=1}^{m-1}r_{j+k}(m+k,i+k)t^jv^{m-1}\right).
\end{align*}
Thus, by \eqref{idr2r3e1}, we get
\begin{align}\label{eq211a2}
R(t;v,u)
&=\frac{vut}{1-t}R(t;v,u)
+tR(t;v,0)+vt^2R(t;v)+\frac{vt^3}{1-2t}\left(R(t;v)-R\left(t;\frac{vt}{1-t}\right)\right)\notag\\
&\quad+\frac{vt^3}{1-(1+v)t}R\left(t;\frac{vt}{1-t}\right)+\frac{v^2t^3}{1-t}R(t;v,1)R\left(t;\frac{vt}{1-t}\right).
\end{align}
Applying now \eqref{eq211a1} to \eqref{eq211a2} yields \eqref{eq211a3}, after some algebra.
\end{proof}

We have that $R(t;v)$ may be expressed in terms of an infinite continued fraction as follows.

\begin{theorem}\label{21-1th}
The generating function $R(t;v)$ for the number of members of $\mathcal{C}_n(21\text{-}1)$ having largest letter $m$ for all $n \geq m \geq1$ (marked by $t^nv^{m-1}$) satisfies
\begin{align}
R(t;v)&=\dfrac{t}{
v(1-2t)-\dfrac{(1-t)^2(vt^2-(1-v)(1-2t)}{(1-t)(1-2t)-vt^3R(t;\frac{vt}{1-t})}}.\label{21-1the1}
\end{align}
In particular,
we have $\sum_{n\geq1}c_n(21\text{-}1)t^n=R(t;1)=t+2t^2+5t^3+13t^4+35t^5+96t^6+267t^7+750t^8+2122t^9+6036t^{10}
+17239t^{11}+49389t^{12}+141842t^{13}+408142t^{14}+1176194t^{15}
+3393726t^{16}+9801731t^{17}+28331971t^{18}+81947397t^{19}+237152119t^{20}+\cdots$.
\end{theorem}
\begin{proof}
Taking $u=0$ into \eqref{eq211a3}, one finds an expression of $R(t;v,0)$, which when substituted back into \eqref{eq211a3}, gives
\begin{align}\label{eq211a4}
\left(1-\frac{vut}{1-t}\right)R(t;v,u)
&=\frac{vt^3}{(1-(1+v)t)((1-t)^2-vt^2)}+\frac{v^2t^3}{(1-t)^2}R(t;v,1)R\left(t;\frac{vt}{1-t},1\right)\notag\\
&\quad-\frac{vt^2(vt^3-(1-t)^3)}{(1-t)(1-2t)((1-t)^2-vt^2)}R(t;v,1)\notag\\
&\quad-\frac{v(1-v)t^4}{(1-t)(1-2t)(1-(1+v)t)}R\left(t;\frac{vt}{1-t},1\right).
\end{align}
Taking $u=1$ in \eqref{eq211a4}, one gets
\begin{align}\label{eq211a5}
&\left(1-\frac{vt}{1-t}+\frac{vt^2(vt^3-(1-t)^3)}{(1-t)(1-2t)((1-t)^2-vt^2)}-\frac{v^2t^3}{(1-t)^2}R\left(t;\frac{vt}{1-t},1\right)\right)R(t;v,1)\notag\\
&=\frac{vt^3}{(1-(1+v)t)((1-t)^2-vt^2)}-\frac{v(1-v)t^4}{(1-t)(1-2t)(1-(1+v)t)}R\left(t;\frac{vt}{1-t},1\right).
\end{align}
Using $R(t;v,1)=R(t;v)-\frac{1}{1-(1+v)t}$ in \eqref{eq211a5} then yields, after several algebraic steps,
\begin{align*}
R(t;v)&=\frac{vt^4R(t;\frac{vt}{1-t})-t(1-t)(1-2t)}{
\frac{v}{t}(1-2t)(vt^4R(t;\frac{vt}{1-t})-t(1-t)(1-2t))+(1-t)^2(vt^2-(1-v)(1-2t))},
\end{align*}
which implies \eqref{21-1the1}.
\end{proof}

\end{document}